\documentclass[a4paper,oneside,10pt,final,reqno]{amsart}

\usepackage{amssymb}
\usepackage[mathcal]{euscript}
\usepackage{mathrsfs}
\usepackage[all]{xy}
\usepackage{enumitem}
\usepackage{bbm}

\usepackage{version,color}
\usepackage[l2tabu,orthodox]{nag}
\usepackage{showkeys}
\usepackage[all, warning]{onlyamsmath}

\numberwithin{equation}{section}
\allowdisplaybreaks

\theoremstyle{definition}
\newtheorem{thm}{Theorem}[subsection]

\newtheorem{lem}[thm]{Lemma}
\newtheorem{dfn}[thm]{Definition}
\newtheorem{fct}[thm]{Fact}
\newtheorem{eg}[thm]{Example}
\newtheorem{rmk}[thm]{Remark}
\newtheorem{asm}[thm]{Assumption}
\newtheorem*{dfn*}{Definition}
\newtheorem*{cor*}{Corollary}
\newtheorem*{fct*}{Fact}
\newtheorem*{eg*}{Example}

\newcommand{\ve}{\varepsilon}
\newcommand{\ol}{\overline}

\newcommand{\wt}{\widetilde}

\newcommand{\sg}{\Sigma}
\newcommand{\sinf}{\tfrac{\infty}{2}}

\newcommand{\inj}{\hookrightarrow}
\newcommand{\surj}{\twoheadrightarrow}
\newcommand{\simto}{\xrightarrow{\sim}}
\newcommand{\longto}{\longrightarrow}
\newcommand{\longsimto}{\xrightarrow{\ \sim \ }}
\newcommand{\longinj}{\lhook\joinrel\longrightarrow}
\newcommand{\longsurj}{\relbar\joinrel\twoheadrightarrow}
\newcommand{\longbij}{\mathrel{\substack{\xrightarrow{\rule{1.5em}{0em}} \\[-.9ex] \xleftarrow{\rule{1.5em}{0em}}}}}

\newcommand{\deq}{\, = \, }

\newcommand{\seteq}{\, := \, }

\newcommand{\bbk}{\mathbbm{k}}
\newcommand{\bbN}{\mathbb{N}}

\newcommand{\bbZ}{\mathbb{Z}}

\newcommand{\one}{\mathbf{1}}

\newcommand{\catA}{\mathcal{A}}
\newcommand{\catB}{\mathcal{B}}
\newcommand{\catC}{\mathcal{C}}
\newcommand{\catD}{\mathcal{D}}
\newcommand{\catE}{\mathcal{E}}
\newcommand{\catM}{\mathcal{M}}

\newcommand{\cOp}{\mathcal{O}p}
\newcommand{\cCo}{\mathcal{C}oop}
\newcommand{\cFun}{\mathop{\mathcal{F}unc}\nolimits}
\newcommand{\cSet}{\mathcal{S}et}

\newcommand{\cMod}[2]{{\mathcal{M}od}^{#1}_#2}
\newcommand{\lMod}[1]{#1\text{-}\mathcal{M}od}

\newcommand{\dgkMod}{\mathrm{dg} \, \lMod{\bbk}}

\newcommand{\frkg}{\mathfrak{g}}

\newcommand{\id}{\mathrm{id}}
\newcommand{\fog}{\mathrm{fog}}

\newcommand{\Der}{\mathop{\mathrm{Der}}\nolimits}

\newcommand{\Hom}{\mathop{\mathrm{Hom}}\nolimits}
\newcommand{\Ind}{\mathop{\mathrm{Ind}}\nolimits}
\newcommand{\Ker}{\mathop{\mathrm{Ker}}\nolimits}
\newcommand{\Coker}{\mathop{\mathrm{Coker}}\nolimits}

\newcommand{\BD}[1]{\mathop{\mathrm{Bar}_{\Delta}^{#1}}\nolimits}

\newcommand{\opC}{\mathsf{C}}
\newcommand{\opF}{\mathsf{F}}
\newcommand{\opI}{\mathsf{I}}
\newcommand{\opP}{\mathsf{P}}
\newcommand{\opQ}{\mathsf{Q}}
\newcommand{\opU}{\mathsf{U}}
\newcommand{\opAsc}{\mathsf{Asc}}
\newcommand{\opCom}{\mathsf{Com}}
\newcommand{\opEnd}{\mathsf{End}}
\newcommand{\opLie}{\mathsf{Lie}}
\newcommand{\opBar}{\mathop{\mathsf{Bar}}\nolimits}
\newcommand{\opCob}{\mathop{\mathsf{Cobar}}\nolimits}

\newcommand{\op}{\text{op}}
\newcommand{\aug}{\text{aug}}

\newcommand{\coaug}{\text{coaug}}

\newcommand{\bcirc}{\mathbin{\overline{\circ}}}

\begin{document}


\title{Operadic semi-infinite homology}

\author{Shintarou Yanagida}
\address{Graduate School of Mathematics, Nagoya University 
Furocho, Chikusaku, Nagoya, Japan, 464-8602.}
\email{yanagida@math.nagoya-u.ac.jp}

\date{March 30, 2018}

\begin{abstract}
We propose the notion of semi-infinite homology for algebras over operads
using the relative homology theory for operadic algebras.
\end{abstract}

\maketitle

\tableofcontents

\setcounter{section}{-1}
\section{Introduction}
\label{sect:intro}

The motivation of this note is to understand 
semi-infinite (co)homology for Lie algebras and associative algebras
with the viewpoint of operads.
Although there are already much literature on semi-infinite homology algebra,
we believe this note sheds new light on the topic.

Let us cite a concise explanation of semi-infinite homology from 
\cite[Introduction]{P1}:
\emph{Roughly speaking, the semi-infinite cohomology is defined for 
a Lie or associative algebra-like object which is split in two halves; 
the semi-infinite cohomology has the features of 
a homology theory (left derived functor) along one half of the variables
and a cohomology theory (right derived functor) along the other half}.

Semi-infinite homology for Lie algebras was first introduced by 
B.\ Feigin in \cite{Fe} in the 1980s.
Voronov proposed in \cite{V1} a homology algebraic treatment
(see also \cite{So,V2}).
For associative algebras,  Arkhipov \cite{A1,A2} first built 
the theory of semi-infinite homology.
Further study is done by Sevostyanov \cite{Se}.

Positselski constructed in \cite{P1} a huge theory using 
semi-infinite version of comodule-contramodule correspondence.
See also \cite{P2} for the theory of derived categories 
in his framework.

In this note 
we develop a relative homology theory for algebras over operads,
and use it to define a standard semi-injective resolution.
The theory of operadic relative (co)homology is built
with the help of cotriple homology, as explained in \S\ref{sect:rel-hom}.
The construction of standard semi-injective resolution 
is explained in \S\ref{sect:std}

Our construction is pretty simple, and may give a transparent 
understanding of various theories of semi-infinite (co)homology.
We plan to apply our method to the algebras 
in the category of Tate vector spaces in future.

\subsection*{Notations}

$\bbN=\{0,1,2,\ldots\}$ denotes the set of non-negative integers.
For $n \in \bbZ_{>0}$, 
we denote by $\sg_n$ the $n$-th symmetric group,
and we set $\sg_0 := \sg_1 = \{e\}$, the trivial group.

For a category $\catA$, we write $X \in \catA$ in the meaning of 
$X$ being an object of $\catA$.
The set of morphisms $X \to Y$ for $X,Y \in \catA$ 
is denoted by $\catA(X,Y)$, or sometimes by $\Hom_{\catA}(X,Y)$.
$\cSet$ denotes the category of ($U$-small) sets 
(if we fix a universe $U$).

We use the word \emph{limit} in the meaning of inverse limit,
and \emph{colimit} in the meaning of direct limit.
The word \emph{dg} means differential $\bbZ$-graded.

Calligraphy symbols like $\catA,\catB,\ldots$ denote categories.
Sans-Serif symbols like $\opP,\opQ,\ldots$ denote operads.

\section{Operads in relative setting}
\label{sect:operad}

Following \cite[Chapters 1--3]{Fr}, 
we give a preliminary on operads in relative setting.
Namely we work over a symmetric monoidal category $\catC$,
which will be called the base category.
Operads $\opP$ we will consider live in this base category $\catC$,
and algebras over $\opP$ will be defined in another category $\catE$ 
which is ``over $\catC$".

Following \cite[\S0.1]{Fr}, we put the following conditions 
on categories and functors.

\begin{asm}\label{asm:colimit}
On categories and functors we assume the following conditions.
\begin{itemize}
\item
Every category $\catA$ considered in this note 
contains a small subcategory $\catA_f$ such that 
every object $X \in \catA$ is the filtered colimit of a diagram of $\catA$.
\item
Every functor $F:\catA \to \catB$ 
preserves filtered colimits.
\end{itemize}
\end{asm}

The assumption on the existence of $\catA_f$ 
implies that a functor $\varphi: \catA \to \catB$ 
admits a right adjoint $\catB \to \catA$ if and only if 
$\varphi$ preserves colimits.
This observation will be used at Definition \ref{dfn:opEnd}
of the endomorphism operad.

\subsection{Operads}
\label{subsec:operads}

We fix a symmetric monoidal category $\catC=(\catC,\otimes,\one)$
satisfying the following conditions.
\begin{enumerate}[label=C\arabic*]
\item
All small colimits and small limits exist.
\item
$\otimes: \catC \times \catC \to \catC$ 
commutes with all small colimits in each variable.
\end{enumerate}
This assumption implies that $\catC$ has the initial object,
which will be denoted by $0$.
We call $\catC$ the base category.

\begin{dfn*}
A \emph{$\sg_*$-object} in $\catC$ 
consists of a sequence $M=\{M(n)\}_{n\in \bbZ}$ 
of objects $M(n) \in \catC$ with a right $\sg_n$-action.
A \emph{morphism $M \to N$ between $\sg_*$-objects} consists of 
a sequence $\{M(n) \to N(n)\}_{n \in \bbZ}$ of morphisms in $\catC$ 
which commute with the $\sg_n$-actions.
We denote by $\catC^{\sg}$ 
the category of $\sg_*$-objects in $\catC$.
\end{dfn*}

The \emph{identity $\sg_*$-object} $\opI \in \catC^{\sg}$ 
is defined by 
\[
 \opI(n) = 
 \begin{cases}
  \one & (n = 1) \\
  0    & (n \neq 1)
 \end{cases}.
\]


The category $\catC^{\sg}$ has a natural monoidal structure
induced by the one in $\catC$.
Before explaining that, let us define the bifunctor 
$\otimes:\catC \times \cSet \to \catC$ by
\[
 C \otimes K := \otimes_{k \in K} C.
\]
Here $\cSet$ denotes the category of sets.
Since $\catC$ has a colimit by the condition C2, 
this definition makes sense.
Note also that it makes sense 
for any category $\catC$ with all small colimits.
Next let $G$ be a group with $m: G \otimes G \to G$ the multiplication map,
$M$ be a left $G$-module with $\lambda: G \otimes M \to M$ the action map, 
and assume that $C \in \catC$ has a right $G$-action $\rho$.
Then the tensor product $C \otimes G \otimes M$ makes sense.
Now we define $C \otimes_{G} M \in \catC$ to be the following coequalizer.
\[
 \xymatrix{
 C \otimes G \otimes M 
 \ar@<0.5ex>[rr]^{\id \otimes \lambda} \ar@<-0.5ex>[rr]_{\rho \otimes \id} 
 & & C \otimes M \ar[r] & C \otimes_{G} M.
 }
\]

We can now explain the monoidal structure on $\catC^{\sg}$.
For $M,N \in \catC^{\sg}$ we set
\[
 (M \otimes N)(n) :=  \bigoplus_{p+q=n} 
  \bigl(M(p) \otimes M(q)\bigr) \otimes_{\sg_p \times \sg_q} \sg_n.
\]
On the right hand side we used the coequalizer definition 
of $\otimes_{\sg_p \times \sg_q}$.
The left action of $\sg_p \times \sg_q$ on $\sg_n$ is defined as follows.
We consider $\sg_p \times \sg_q$ as a subgroup of $\sg_n$ 
by identifying $\sigma \in \sg_p$ with a permutation of 
$\{1,\ldots,p\} \subset \{1,\ldots,p,p+1,\ldots,n\}$
and by identifying $\tau \in \sg_q$ with a permutation of 
$\{p+1,\ldots,n\} \subset \{1,\ldots,p,p+1,\ldots,n\}$.
Then $\sg_p \times \sg_q$ acts on $\sg_n$ by translations on the right,
which is the desired action.
Now the right action of $\sg_n$ on itself makes $(M \otimes  N)(n)$ 
a right $\sg_n$-module.
Thus we have a $\sg_*$-object $M \otimes N = \{(M \otimes N)(n)\}_{n\in\bbN}$.
One can check that the triple
\[
 (\catC^{\sg}, \otimes, \opI)
\]
is a symmetric monoidal category.

The category $\catC^{\sg}$ 
is equipped with another monoidal structure 
$(\catC^{\sg}, \circ, \opI)$, 
where the monoidal operation $\circ$ is defined by the following.

\begin{dfn}\label{dfn:circ}
For $M, N \in \catC^{\sg}$, define $M \circ N \in \catC^{\sg}$ by
\[
 (M \circ N)(r) \seteq
 \bigoplus_{k \ge 0} \bigl(M(k) \otimes N^{\otimes k} (r) \bigr)_{\sg_k}.
\]
In the right hand side the group $\sg_k$ acts diagonally on
$M(k) \otimes N^{\otimes k} (r)$,
where the action on $N^{\otimes k} (r)$ is the permutation of $k$ factors.
The symbol $X_{G}$ denotes the coinvariants 
in $X \in \catC$ with respect to the action of a group $G$.
\end{dfn}

Since we assume that $\catC$ has colimits, and that 
$\otimes$ commutes with colimits,
this definition makes sense and 
$M\circ N$ is actually a $\sg_*$-object.
More explicitly,
the $r$-th component $(M \circ N)(r)$ is given by 
\[
 (M \circ N)(r) = 
 \bigoplus_{k \ge 0} M(k) \otimes_{\sg_k} 
 \Bigl(
  \bigoplus \Ind^{\sg_r}_{\sg_{i_1} \times \cdots \times \sg_{i_k}}
  N(i_1) \otimes \cdots \otimes N(i_k)
 \Bigr).
\]
Here the second direct sum runs over 
$(i_1,\ldots,i_k) \in \bbN^k$ such that $i_1+\cdots+i_k=n$.
The right $\sg_k$-action on the right-hand side factor 
is the permutation of the $k$-tuple $(i_1,\ldots,i_k)$.
This presentation makes the natural $\sg_r$-action on $(M\circ N)(r)$ explicit.
See \cite[\S5.1.4]{LV} for a detailed explanation.

In a set-theoretic context, 
we denote an element of 
\[
 M(k) \otimes_{\sg_k} 
 \Ind^{\sg_r}_{\sg_{i_1} \times \cdots \times \sg_{i_k}}
 N(i_1) \otimes \cdots \otimes N(i_k) 
 \ \subset \  
 (M \circ N)(r)
\]
by the following notation \cite[\S5.1.7]{LV}.
\begin{equation}\label{eq:compos:elem}
 (m;n_1,\ldots,n_k;\sigma), \quad 
 m \in M(k), \ 
 n_j \in N(i_j), \ 
 \sigma \in Sh(i_1,\ldots,i_k).
\end{equation}
Here $Sh(i_1 \ldots,i_k)$ is the set of $(i_1 \ldots,i_k)$-shuffles.
In other words, $\sigma \in \sg_r$ is an $(i_1 \ldots,i_k)$-shuffle
if it satisfies 
\begin{align*}
&\sigma(1)<\sigma(2)<\cdots<\sigma(i_1), \ 
 \sigma(i_1+1)<\cdots<\sigma(i_1+i_2), \ \ldots,
\\
&\ldots, \ 
 \sigma(i_1+\cdots+i_{k-1}+1)<\cdots<\sigma(i_1+\cdots+i_k)=\sigma(r).
\end{align*}
The set $Sh(i_1 \ldots,i_k)$ is bijective to the quotient 
$(\sg_{i_1} \times \cdots \times \sg_{i_k})\backslash \sg_r$.

\begin{dfn*}
An \emph{operad in $\catC$} is a triple
\[
 (\opP,\mu,\eta)
\]
consisting of a $\sg_*$-object $\opP$ in $\catC$ 
and two morphisms of $\sg_*$-objects 
$\mu: \opP \circ \opP \to \opP$ and $\eta: \opI \to \opP$
making the following diagrams commutative.
\[
 \xymatrix{
  \opP \circ \opP \circ \opP 
   \ar[r]^(0.55){\mu \circ \opP} \ar[d]_{\opP \circ \mu} & 
  \opP \circ \opP \ar[d]^{\mu} \\
  \opP \circ \opP \ar[r]_{\mu} &
  \opP
 }
 \qquad
 \xymatrix{
  \opI \circ \opP \ar[r]^{\eta \circ \id} \ar[rd]_{\id} & 
  \opP \circ \opP \ar[d]^(0.45){\mu} & 
  \opP \circ \opI \ar[l]_(0.45){\id \circ \eta} \ar[ld]^{\id}  \\
  & \opP
 } 
\]
The morphism $\mu$ is called the \emph{composition map},
and $\eta$ is called the \emph{unit map}.
\end{dfn*}

One can see that the identity $\sg_*$-object $\opI$ 
has an obvious operad structure.
This operad is called \emph{the identity operad} 
and denoted by the same symbol $\opI$.

The composition map $\mu$ is encoded by the \emph{partial compositions}  
\begin{align}\label{eq:circ_i}
 \circ_i: \opP(r) \otimes \opP(s) \longto \opP(r+s-1)
\end{align}
for $1 \le i \le r$.
These are defined to be the composites
\begin{align*}
 \opP(r) \otimes \opP(s) 
 \longsimto
&\opP(r) \otimes 
 (\one \otimes \cdots \otimes \overset{\text{$i$-th}}{\opP(s)} 
       \otimes \cdots \otimes \one)
\\
 \longto
&\opP(r) \otimes 
 (\opP(1) \otimes \cdots \otimes \opP(s) \otimes \cdots \otimes \opP(1))
 \xrightarrow{\ \mu \ }
 \opP(r+s-1),
\end{align*}
where the second morphism is 
$\id \otimes (\eta \otimes \cdots \otimes \id \otimes \cdots \otimes \eta)$.
See \cite[\S5.3.4]{LV} for the equivalence of the definitions of 
operads in terms of the composition map $\mu$ and 
the partial composition maps $\circ_i$.

\begin{dfn*}
A \emph{morphism $\varphi: \opP \to \opQ$ between operads} in $\catC$ 
is defined to be a morphism of $\sg_*$-objects that preserves operad structures.
Such $\varphi$ will be called operad morphism.
We denote by $\cOp_{\catC}$ the category of operads in $\catC$.
\end{dfn*}

Let us also recall the notion of free operads.
The forgetful functor $\cOp_{\catC} \to \catC^{\sg}$ from 
operads to $\sg_*$-objects has a left adjoint
\[
 \opF: \catC^{\sg} \longto \cOp_{\catC}
\]
making each object $M$ to an operad $\opF(M)$ in $\catC$.


Explicitly, the underlying $\sg_*$-object of $\opF(M)$ 
is given by
\[
 \opF(M) \seteq \varinjlim \opF_n(M),
\]
where $\opF_n(M)$ is determined inductively by
\[
 \opF_0(M) \seteq \opI, \quad
 \opF_n(M) \seteq M \circ \opF_{n-1}(M),
\]
and the inclusion map $i_n:\opF_n(M) \inj \opF_{n+1}(M)$
is given by $i_1: \opI \inj \opI \oplus M$, inclusion in the first factor,
and  $i_n := \id \oplus (\id \circ i_{n-1})$.
Any induced inclusion map $\opF_m(M) \inj \opF_n(M)$ is denoted by $i$.
We also have the inclusion map $M \circ \opF_{n-1}(M) \inj \opF_n(M)$
into second factor.
The induced map $M \inj \opF(M)$ is denoted by $j$.

The composition map $\mu$ on $\opF(M)$ is determined by
$\mu_{m,n}: \opF_m(M) \circ \opF_n(M) \to \opF_{m+n}(M)$,
and we construct $\mu_{m,n}$ inductively.
First we set $\mu_{0,n}:= \id$ on $\opF_0(M) \circ \opF_n(M) \to \opF_n$.
Then we define $\mu_{m,n}$ to be the following composite.
\begin{align*}
 \opF_m(M) \circ \opF_n(M)
&\longsimto 
 \opF_n(M) \oplus \bigl(M \circ (\opF_{m-1}(M) \circ \opF_n(M)) \bigr) \\
&\longsimto
 \opF_n(M) \oplus \bigl((M \circ \opF_{m-1}(M)) \circ \opF_n(M) \bigr)
\\
&\xrightarrow{\, (\id,\id \circ \mu_{m-1,n}) \, }
 \opF_n(M) \oplus M \circ \opF_{m+n-1}(M)
 \xrightarrow{\, i+j \, }
 \opF_{m+n}(M).
\end{align*}

We can also define the unit map $\eta:\opI \to \opF(M)$ to be the composite 
$\opI \simto \opF_0(M) \inj \opF(M)$.
The triple $(\opF(M),\mu,\eta)$ obtained in this way is indeed an operad.
See \cite[\S5.5.1]{LV} for the proof.

\begin{dfn}\label{dfn:free-operad}
The operad $\opF(M)$ is called the \emph{free operad} associated to $M$.
\end{dfn}

The natural projection $\ve:\opF(M) \to \opI$ is an \emph{augmentation map}
of the free operad $\opF(M)$, i.e., 
$\ve$ is a morphism of operads to the identity operad
satisfying the relation $\ve \eta = \id$.
Thus $\opF(M)$ is an augmented operad in the following sense.

\begin{dfn}\label{dfn:augmented-operad}
An operad $\opP$ is called \emph{augmented} 
if it has the augmentation map $\ve:\opP \to \opI$.
The kernel $\Ker(\ve)$ is called the \emph{augmentation ideal}
and denoted by $\opP_{\aug}$.
\end{dfn}

Concrete examples of operads are constructed via generators and relations.
Let $M,F \in \catC^{\sg}$ be $\sg_*$-objects.
Then we can define an operad $\opP$ in $\catC$ to be a coequalizer 
of the form 
\[
 \xymatrix{
 \opF(R) \ar@<0.5ex>[r]^{d_0} \ar@<-0.5ex>[r]_{d_1} & 
 \opF(M) \ar[r] & \opP.
 }
\]

In the classical setting like $\catC = \lMod{k}$, 
the category of $k$-modules over a commutative ring $k$,
one usually take $R \subset \opF(M)$ 
and consider $\opP :=\opF(M)/(R)$.
Here $(R)$ denotes the ideal generated by $R$ in the operadic sense,
and it is spanned by composites which include a factor of the form 
$w \rho$ with $w \in \sg_n$ and $\rho \in R(n)$.
In terms of the coequalizer definition, 
one takes $d_1=0$ to get this operad $\opP$.
See \cite[Chapter 5]{LV} for the detail.

Closing this subsection,
let us recall the classical examples of operads, 
namely the associative, commutative and Lie operads.

\begin{eg}\label{eg:classical-operads}
Let $\catC = (\lMod{k},\otimes,k)$ be the category of $k$-modules 
over a commutative ring $k$
with $\otimes=\otimes_k$ the ordinary tensor product.
Then
\begin{enumerate}
\item 
The commutative operad $\opCom$ is 
\[
 \opCom = \opF(k \mu)/(\tau \mu - \mu, \mu \circ_1 \mu - \mu \circ_2 \mu).
\]
Thus the generating $\sg_*$-object 
$M=k \mu$ is spanned by a binary operation $\mu=\mu(x_1,x_2)$, 
and as a $k$-module $F(k \mu)$ is isomorphic to the tensor algebra
$\oplus_{n\ge0} (k \mu)^{\otimes n}$.
The defining relations are given by two elements 
$\tau \mu - \mu \in \opF(k \mu)(2)$ and 
$\mu \circ_1 \mu- \mu \circ_2 \mu \in \opF(k \mu)(3)$.
Here $\tau=(1,2)$ denotes the permutation.
Thus the relation means the commutativity 
$\mu(x_1,x_2)=\mu(x_2,x_1)$ 
and the associativity
$\mu(\mu(x_1,x_2),x_3)=\mu(x_1,\mu(x_2,x_3))$ 
of the multiplication $\mu$.

\item
The associative operad $\opAsc$ is 
\[
 \opCom = \opF(k \mu \oplus k \tau \mu)/(\mu \circ_1 \mu - \mu \circ_2 \mu).
\]

\item
The Lie operad $\opLie$ is
\[
 \opLie = \opF(k \mu)/(\mu+\tau\mu, (1+c+c^2)\mu \circ_1 \mu).
\]
Here we denoted $c=(1,2,3) \in \sg_3$.
The first relation means the anti-commutativity of the operation 
$\mu=\mu(x_1,x_2)$,
and the second means the Jacobi identity.
\end{enumerate}
We also have a sequence of morphisms of operads
\begin{equation}\label{eq:classical-operads-seq}
 \opLie \longto \opAsc \longto \opCom.
\end{equation}
The first morphism is induced by the natural embedding 
of $\sg_n$-modules $\opLie(n)$ 
into the regular $\sg_n$-modules $\opAsc(n)$.
The second one is induced by the natural quotient.
\end{eg}

\subsection{Algebras over operads}

Let us fix a symmetric monoidal category $\catC$
with the same assumptions as in the previous subsection.
Following \cite[Chapter 1]{Fr} we introduce 
the notion of symmetric monoidal category over $\catC$

\begin{dfn}\label{dfn:E:over:C}
A \emph{symmetric monoidal category $\catE$ over $\catC$} is 
$\catE=(\catE,\otimes,\one)$ 
together with a bifunctor $\otimes: \catC \times \catE \to \catE$
such that 
\begin{enumerate}[label=E\arabic*]
\item
$\one_{\catC} \otimes X \simeq X$ for all $X \in \catE$,
\item
$(C \otimes D) \otimes X \simeq C \otimes (D \otimes X)$
for all $C,D \in \catC$ and $X \in \catE$,
\item
$C \otimes (X \otimes Y) \simeq (C \otimes X) \otimes Y
 \simeq X \otimes (C \otimes Y)$ 
for all $C \in \catC$ and $X,Y \in \catE$.
\end{enumerate}
At the first condition we wrote $\one_{\catC}$ 
the unit object of the monoidal category 
$\catC=(\catC,\otimes,\one_{\catC})$.
This bifunctor $\otimes: \catC \times \catE \to \catE$ 
is called the \emph{external tensor product},
and $\otimes: \catE \times \catE \to \catE$ 
is called the \emph{internal tensor product}.
We also assume the following conditions.
\begin{enumerate}[label=E\arabic*]
\setcounter{enumi}{3}
\item
All small colimits and small limits exist in $\catE$.
\item
The internal tensor product $\otimes: \catE \times \catE \to \catE$
commutes with small colimits in each variable.
\item
The external tensor product $\otimes: \catC \times \catE \to \catE$
commutes with small colimits in each variable.
\end{enumerate}
\end{dfn}

A symmetric monoidal category $\catE$ will be the place 
where an algebra over an operad lives.
For the definition of algebras over operads, we need one more preparation.

\begin{dfn*}
For a $\sg_*$-object $P$ in $\catC$ and an object $E \in \catE$,
we define 
\[
 S(P,E) := 
 \bigoplus_{n \ge 0} \bigl(P(n) \otimes E^{\otimes n}\bigr)_{\sg_n} 
 \, \in \catE.
\]
\end{dfn*}

Here we used the internal tensor product to form $E^{\otimes n}$,
and used the external tensor product to form $P(n) \otimes E^{\otimes n}$.
The existence of colimits assures that the coinvariants
$(P(n) \otimes X^{\otimes n})_{\sg_n}$ is well-defined.
We see from the definition that 
if $\catE=\catC$ and the external and internal tensor products coincide,
then $S(P,E) = P \circ E$.

The operation $P \to S(P,-)$ has the following nice property.
Let $\cFun(\catE,\catE)$ be the category of functors on $\catE$.

\begin{fct*}
The operation $P \to S(P,-)$ defines a functor
\[
 (\catC^{\sg_*},\circ,\opI) \longto (\cFun(\catE,\catE),\circ,\id)
\]
of monoidal categories.
Here the $\circ$ in the target denotes the composition of functors.
\end{fct*}

\begin{dfn*}
Let $\opP$ be an operad in $\catC$.
A \emph{$\opP$-algebra in $\catE$} is a pair 
\[
 (A,\mu_A) 
\]
consisting of an object $A \in \catE$
and a morphism $\mu_A: S(\opP,A) \to A$ in $\catE$ 
making the following diagrams commutative.
\begin{align*}
 \xymatrix{
  S(\opP \circ \opP,A) \ar[r]^(0.45){\sim} \ar[d]_{S(\mu,A)} & 
  S(\opP,S(\opP,A)) \ar[rr]^(0.6){S(\opP,\mu_A)} & & 
  S(\opP,A) \ar[d]^{\mu_A} \\
  S(\opP,A) \ar[rrr]_{\mu_A} & & & A 
 } 
 \qquad
 \xymatrix{
  S(\opI,A) \ar[r]^{S(\eta,\id)} \ar[dr]_{\sim}& 
  S(\opP,A) \ar[d]^{\mu_A} \\
  & A
 }
\end{align*}
The morphism $\mu_A$ is called the \emph{evaluation map}.
We denote such $(A,\mu_A)$ simply by $A$.
\end{dfn*}

More explicitly,
$\mu_A$ consists of a collection of morphisms
$\opP(n) \otimes A^{\otimes n} \to A$
which are equivariant with respect to the $\sg_n$-actions,
and which also satisfy the unit and associative conditions 
in terms of operad actions.

One can restate $\opP$-algebra structures on $A \in \catE$ 
in terms of endomorphism operads.
See \cite[\S5.2.11]{LV} for an explanation in the non-relative setting.
{}To explain it in the relative setting, 
let us introduce the endomorphism operad 
$\opEnd_A$ for $A \in \catE$ following \cite[\S3.4]{Fr}.

Recall the assumption that the external tensor product 
$\otimes: \catC \times \catE \to \catE$ preserves colimits.
Then the argument after Assumption \ref{asm:colimit} 
says that there is  a bifunctor
\[
 \Hom_{\catE}: 
 \catE^{\op} \times \catE \longto \catC
\]
such that for any $C \in \catC$ and $X,Y \in \catE$ we have 
\begin{equation}\label{eq:HomE}
 \catE(C \otimes X, Y) \deq \catC(C,\Hom_{\catE}(X,Y)).
\end{equation}

\begin{dfn}\label{dfn:opEnd}
Let $X \in \catE$.
The \emph{endomorphism operad} $\opEnd_X$ of $X$ 
is the $\sg_*$-object in $\catC$ consisting of 
\[
 \opEnd_X(r) := \Hom_{\catE}(X^{\otimes r}, X),
\] 
where $\sg_r$-acts  
by permuting the tensor components of $X^{\otimes r}$,
together with the composition map $\mu$ 
given by the natural composition of endomorphisms.
\end{dfn}

Using the set-theoretic symbol \eqref{eq:compos:elem},
we can write the action of $\mu$ as 
\[
 \mu(f;f_1,\ldots,f_k;1) = f(f_1 \otimes \cdots \otimes f_k)
\]
for $f \in \Hom_{\catE}(X^{\otimes k},X)$ 
and $f_j \in \Hom_{\catE}(X^{\otimes i_j},X)$.
In the case $\sigma \neq \id$, 
the right $\sigma_r$-action and 
the symmetric monoidal structure of $\catE$ over $\catC$
determines $\mu(f;f_1,\ldots,f_k;\sigma)$. 

Now we have the well-known equivalence
between $\opP$-algebra structures on a given object $A$ and 
operad morphisms between $\opP$ and $\opEnd_A$.

\begin{fct}[{\cite[Proposition 3.4.3]{Fr}}]
\label{fct:P-alg=op-mor}
For any $A \in \catE$,
there is a one-to-one correspondence between operad morphisms 
$\opP \to \opEnd_{A}$ and $\opP$-algebra structures on $A$.
\end{fct}

\begin{proof}
By  \eqref{eq:HomE} we have 
$\catC(\opP(n),\opEnd_A(n)) = \catE(P(n)\otimes A^{\otimes n},A)$.
Thus we have 
\[
 \catC^{\sg}(\opP,\opEnd_{A}) \deq 
 \catE\bigl(\oplus_{n\ge0} (\opP(n) \otimes A^{\otimes n})_{\sg_n},A \bigr) 
 \deq \catE(S(P,A),A),
\]
which yields the statement.
\end{proof}

\begin{dfn}\label{dfn:P-alg:hom}
A \emph{morphism $f: A \to B$ between $\opP$-algebras} $A$ and $B$ in $\catE$
is a morphism in $\catE$ commuting with the evaluation maps $\mu_A$ and $\mu_B$.
In other words, the following diagram commutes.
\begin{align}\label{diag:P-alg:hom}
 \xymatrix{
  S(\opP,A) \ar[r]^(0.6){\mu_A} \ar[d]_{S(\id,f)} & A \ar[d]^{f} \\
  S(\opP,B) \ar[r]_(0.6){\mu_B} & B
 }
\end{align}
The category of $\opP$-algebras in $\catE$ is denoted by ${}_{\opP}\catE$.
\end{dfn}

\begin{eg}\label{eg:classical-op-alg}
Algebras over the classical operads in Example \ref{eg:classical-operads}
are commutative, associative and Lie algebras in the ordinary meaning.
Precisely speaking, setting $\catE = \catC$ to be the category of $k$-modules
over a commutative ring $k$, we have the followings.
\begin{enumerate}
\item
A $\opCom$-algebra $C$ is a commutative $k$-algebra $(C,\mu)$.
\item
An $\opAsc$-algebra $A$ is an associative $k$-algebra $(A,\mu)$.
\item
A $\opLie$-algebra $\frkg$ is a Lie algebra $(\frkg, [\cdot , \cdot]=\mu)$
over $k$.
\end{enumerate}
The sequence \eqref{eq:classical-operads-seq} says that 
for $A \in \catC$
a commutative algebra structure on $A$ naturally gives 
an associative algebra structure,
and an associative algebra structure on $A$ gives rise to 
a Lie algebra structure by setting the Lie bracket to be the commutator.
\end{eg}

\subsection{Modules of algebras over operads}

Following \cite[\S4.2]{Fr},
we introduce modules over a $\opP$-algebra 
in the relative setting.
Note that in \cite{Fr} they are called 
representations of a $\opP$-algebra.
See \cite[\S12.3.1]{LV} for the non-relative setting.

Fix $\catC$ and $\catE$ as in the previous subsections.
Recall that $\catC^{\sg}$ denotes 
the category of $\sg_*$-objects in $\catC$.
For $P \in \catC^{\sg}$ and $M,N \in \catE$,
define $S(P,M;N) \in \catE$ to be 
\[
 S(P,M;N) := 
  \bigoplus_{k \ge 0} \Bigl( P(k) \otimes 
  \bigl(\bigoplus_{1 \le i \le k} M \otimes \cdots \otimes 
   \overset{\text{$i$-th}}{N} \otimes \cdots \otimes M \bigr) \Bigr)_{\sg_k}
\] 
where in the right hand side $M$ appears $k-1$ times.
The symbol $(-)_{\sg_k}$ denotes the coinvariants as before.
Thus we have $S(P,M;M)=S(P,M)$.
Note that for $P,Q \in \catC^{\sg}$ and $M,N \in \catE$
we have a natural isomorphism
\[
 S(P,S(Q,M);S(Q,M;N)) \simeq
 S(P \circ Q, M;N).
\]
We also have an isomorphism
\[
 S(\opI,M;N) = \one \otimes N \simeq N
\]
since $\opI(1)=\one$ and $\opI(n)=0$ for $n \neq 1$.

\begin{dfn}\label{dfn:mod-P-alg}
Let $\opP=(\opP,\mu,\eta)$ be an operad in $\catC$,
and $A=(A,\mu_A)$ be a $\opP$-algebra in $\catE$.
An \emph{$A$-module} (over $\opP$ in $\catE$) is a pair 
\[
 (E,\lambda_E) 
\]
consisting of $E \in \catE$ and a morphism
$\lambda_E: S(\opP,A;E) \to E$ 
in $\catE$ making the following diagrams commutative.
\begin{align*}
&\xymatrix{
 S(\opP \circ \opP,A;E) \ar[d]_{S(\mu,\id;\id)} &
 S(\opP,S(\opP,A);S(\opP,A;E)) 
  \ar[l]_(0.57){\sim} \ar[rr]^(0.61){S(\id,\mu_A;\lambda_E)} & & 
 S(\opP,A;E) \ar[d]^{  \lambda_E} \\
 S(\opP,A;E) \ar[rrr]_{\lambda_E} & & & E
}
\\ 
&\hskip 9em
 \xymatrix{
 S(\opI,A;E) \ar[rr]^{S(\eta,\id;\id)} \ar[rrd]_{\sim} & & 
 S(\opP,A;E) \ar[d]^{ \lambda_E} \\
 & & E
}
\end{align*}
The morphism $\lambda_E$ is called \emph{the action map}.
We denote such $(E,\lambda_E)$ simply by $E$.

A \emph{morphism $f: E \to E'$ of $A$-modules} is 
a morphism in $\catE$ commuting with the action maps 
$\lambda_E$ and $\lambda_{E'}$.

We denote by $\cMod{\opP}{A}$ the category of $A$-modules in $\catE$.
\end{dfn}

Let us explain a universal construction of $A$-modules,
namely the construction of \emph{free $A$-module} generated by an object in $\catE$.

Given $M \in \catE$
we define $A \otimes^{\opP} M \in \catE$ to be the coequalizer 
\[
 \xymatrix{
  S(\opP,S(\opP,A);M) 
  \ar@<0.5ex>[r]^(0.55){\wt{\mu}} \ar@<-0.5ex>[r]_(0.55){\wt{\mu}_A} &
  S(\opP,A;M) \ar[r] & A \otimes^{\opP} M
 },
\]
where $\wt{\mu}$ and $\wt{\mu}_A$ are given by 
\[
 \wt{\mu}: \,
 S(\opP,S(\opP,A);M) \longsimto S(\opP \circ \opP,A;M) 
 \xrightarrow{\, S(\mu,\id;\id) \,} S(\opP,A;M)
\]
and
\[
 \wt{\mu}_A: \, 
 S(\opP,S(\opP,A);M) \xrightarrow{\, S(\id,\mu_A;\id) \,} S(\opP,A;M).
\]
$A \otimes^{\opP} M$ is actually an $A$-module.
The $A$-action on $A \otimes^{\opP} M$ is given by the following fact: 
the composite 
\[
 S(\opP,A;S(\opP,A;M)) \longto S(\opP \circ \opP,A;M)
 \xrightarrow{\, S(\mu,\id;\id) \, } S(\opP,A:M)
\]
factors through the quotient
$S(\opP,A;M) \surj A \otimes^{\opP}M$.
See \cite[Lemma 12.3.3]{LV} for the proof.

\begin{dfn}\label{dfn:free-A-mod}
Let $A \in {}_{\opP}\catE$ and $M \in \catE$.
The $A$-module $A \otimes^{\opP}M$ is called 
\emph{the free $A$-module} generated by $M$. 
\end{dfn}

One can also check that 
\[
 A \otimes^{\opP} - : \, \catE \longto \cMod{\opP}{A}
\]
is a functor, and moreover it is a left adjoint of 
the forgetful functor $\fog: \cMod{\opP}{A} \to \catE$.
In other words, we have 

\begin{fct*}[{\cite[Theorem 12.3.4]{LV}}]
For 
any $A \in {}_{\opP}\catE$, 
the two functors
\[
 A \otimes^{\opP} -: \, 
 \catE \longbij \cMod{\opP}{A} \, :\fog
\]
form an adjoint pair.
Thus for any $M \in \cMod{\opP}{A}$ and $N \in \catE$ we have
\[
 \cMod{\opP}{A}(A \otimes^{\opP}M,N) = 
 \catE(M,\fog(N)).
\]
\end{fct*}

Another example of $A$-modules is given by 

\begin{lem}\label{lem:mod-by-alg}
Let $f: A \to B$ be a morphism in ${}_{\opP}\catE$.
Then $B$ has a natural structure of $A$-module.
\end{lem}

\begin{proof}
Define $\lambda_B:S(\opP,A;B) \to B$ to be the composition 
\[
 \lambda_B := \bigl(S(\opP,A;B) \xrightarrow{S(\id,f;\id)} 
 S(\opP,B;B) = S(\opP,B) \xrightarrow{\mu_B} B\bigr).
\]
By the commutativity of \eqref{diag:P-alg:hom},
we have the following commutative diagram.
\[
 \xymatrix{
  S(\opP,A;A) \ar[r]^{\id} \ar[d]_{S(\id,\id;f)} 
  & S(\opP,A) \ar[r]^(0.6){\mu_A} & A \ar[d]^{f} \\
  S(\opP,A;B) \ar[rr]_(0.6){\lambda_B} & & B
 }
\]
Using this diagram repeatedly, one can check that the two diagrams 
in Definition \ref{dfn:mod-P-alg} commute.
\end{proof}

In particular, $A$ is itself an $A$-module
whose action map is given by $\lambda_A = \mu_A$.

\subsection{Enveloping Operads}

Following \cite[\S4.1]{Fr} 
we recall the enveloping operads for an operad.
Let $\catC$ and $\catE$ be symmetric monoidal categories 
as in the previous subsections.
Recall that $\cOp_{\catE}$ denotes the category of operads in $\catE$.

We have a symmetric monoidal functor 
$\eta: (\catC, \otimes,\one) \to (\catE,\otimes,\one)$
given by $\eta(C) = C \otimes \one$,
where $C \otimes \one$ denotes the external tensor product.
Using $\eta$, we can map an operad in $\catC$ to an operad in $\catE$.

Let $\opP$ be an operad in $\catC$.
Now take $\catE$ to be a symmetric monoidal category over $\catC$,
and consider the comma category $\opP/\cOp_{\catE}$ 
of objects under $\opP$.
More explicitly, we have 

\begin{dfn*}
Consider the category whose objects are operad morphisms 
$\phi: \opP \to \opQ$ in $\catE$
with $\opP$ regarded as an operad in $\catE$ via the functor $\eta$, 
and whose morphisms are commutative diagrams
\[
 \xymatrix{ 
  & \opP \ar[ld]_{\phi} \ar[rd]^{\phi'} \\
  \opQ \ar[rr] & & \opQ' }
\]
We denote it by $\opP/\cOp_{\catE}$. 
\end{dfn*}

Using this category $\opP/\cOp_{\catE}$
we can give a universal definition of enveloping operad.

\begin{dfn*}
The \emph{enveloping operad} of a $\opP$-algebra $A$ 
is an operad $\opU_{\opP} \in \opP/\cOp_{\catE}$ 
defined by the following adjunction.
\[
 (\opP/\cOp_{\catE})(\opU_{\opP}(A),\opQ) \deq
 {}_{\opP}\catE(A,\opQ(0)).
\]
\end{dfn*}

The right hand side makes sense 
since $\opQ(0)$ is the initial object in the category ${}_{\opQ}\catE$ 
and the structure map $\opP \to \opQ$ makes $\opQ(0)$ a $\opP$-algebra.
The existence of $\opU_{\opP}(A)$ follows from the fact 
that the functor $\opQ \mapsto \opQ(0)$ preserves limits.

In a set-theoretic context,
we have the following description of $\opU_{\opP}(A)(m)$.
It is spanned by formal elements
\[
 u(x_1,\ldots,x_m) = p(x_1,\ldots,x_m,a_1,\ldots,a_r),
\]
where $p \in \opP(m+r)$ with arbitrary $r \in \bbN$, 
$x_i$ are formal variables and $a_i \in A$.
These elements satisfy the relations of the form
\begin{align*}
&p(x_1,\ldots,x_m,a_1,\ldots,a_{e-1},q(a_e,\ldots,a_{e+s-1}),
   a_{e+s},\ldots,a_{r+s-1})
\\
&= p \circ_{m+e} q(x_1,\ldots,x_r,a_1,\ldots,a_{r+s-1}).
\end{align*}


By taking the unary part of the enveloping operad,
we get an algebra object in the monoidal category $\catE$.
This is the operadic enveloping algebra.

\begin{dfn}\label{dfn:env:op-alg}
The \emph{enveloping algebra} $U_{\opP}(A)$ of a $\opP$-algebra $A$ 
is a unital associative ring object $\opU_{\opP}(A)(1)$ in $\catE$ 
whose multiplication 
$U_{\opP}(A) \otimes U_{\opP}(A) \to U_{\opP}(A)$
is induced by the composition map $\mu_A$.
\end{dfn}

Using the set-theoretical description of $\opU_{\opP}(A)$,
we can consider the universal algebra $U_{\opP}(A)$ 
as the ring with the generators 
\[
 u = p(x,a_1,\ldots,a_n), \quad p \in \opP(1+n), \ a_i \in A
\]
and the defining relations
\begin{align*}
&p(x,a_1,\ldots,a_{e-1},q(a_e,\ldots,a_{e+n-1}),
   a_{e+n},\ldots,a_{m+n-1})
\\
&= p \circ_{1+e} q(x,a_1,\ldots,a_{m+n-1}).
\end{align*}

We have a more explicit description using the free $A$-module
in Definition \ref{dfn:free-A-mod}.

\begin{fct*}[{\cite[\S12.3.4]{LV}}]
For $\opP=(\opP,\mu,\eta) \in \cOp_{\catC}$ and $A \in {}_{\opP}\catE$,
the unital associative ring $U_{\opP}(A)$ is isomorphic to 
\[
 (A \otimes^{\opP} \one_{\catE},m,u),
\]
where the multiplication $m$ is given by
\[
 m: \, (A \otimes^{\opP} \one) \otimes (A \otimes^{\opP} \one)
       \longsimto A \otimes^{\opP}(A \otimes^{\opP} \one)
       \longto A \otimes^{\opP} \one
\]
with the second morphism induced by the composition map $\mu$,
and where the unit $u:\one \to A \otimes^{\opP} \one$ is induced by 
the unit map $\eta$. 
\end{fct*}

As in the case of the ordinary Lie algebras, we have

\begin{fct}[{\cite[4.3.2 Proposition]{Fr}}]\label{fct:rep=mod}
Let $A$ be a $\opP$-algebra in $\catE$.
The category of $A$-modules in $\catE$ 
is equivalent to the category of left $U_{\opP}(A)$-modules.
\end{fct}

\begin{eg}\label{eg:env-alg}
For the classical operads in Example \ref{eg:classical-operads}
and Example \ref{eg:classical-op-alg},
we have the followings.
\begin{enumerate}
\item
For the commutative operad $\opP=\opCom$, 
the enveloping algebra $U_{\opCom}(A)$ is 
the unitary commutative algebra $A_+ = \one \oplus A$.
\item
For the associative operad $\opP=\opAsc$, 
the enveloping algebra $U_{\opAsc}(A)$ is 
the classical enveloping algebra $A^e = A_+ \otimes A_+^{\op}$.
\item
For the Lie operad $\opP=\opLie$,
the enveloping algebra $U_{\opLie}(\frkg)$ is 
the classical enveloping algebra $U(\frkg)$ 
if we take the base ring $k$ to be a field.
\end{enumerate}
\end{eg}

Since the enveloping operad is defined by adjunction,
the correspondence $A \mapsto U_{\opP}(A)$ 
enjoys a functoriality.

\begin{fct}[{\cite[Proposition 12.3.9]{LV}}]
\label{fct:UP-functor}
For any operad $\opP$ in $\catC$,
the enveloping algebra construction gives a functor 
\[
 U_{\opP}: \ 
 {}_{\opP}\catE \longto 
 (\text{unital associative algebra objects in $\catE$}).
\]
\end{fct}

We close this subsection by the recollection on 
the relative free module \cite[\S12.3.5]{LV}.
First we note that 
any morphism $f: B \to A$ of $\opP$-algebras induces a functor
\[
 f^*: \, \cMod{\opP}{A} \longto \cMod{\opP}{B}.
\]
Indeed, by Fact \ref{fct:UP-functor},
we have a morphism $U_{\opP}(f): U_{\opP}(B) \to U_{\opP}(A)$ 
of associative algebras, 
and it gives the functor $U_{\opP}(f)^*$ from 
the category of left $U_{\opP}(A)$-modules to 
the category of left $U_{\opP}(B)$-modules 
by the restriction of scalar.
Then by Fact \ref{fct:rep=mod} we have the desired functor $f^*$.

\begin{dfn*}
We call $f^*$ \emph{the restriction functor}.
\end{dfn*}

\begin{rmk}\label{rmk:AB}
Let us note that 
there is a natural $A$-module structure on $A$ itself
by Lemma \ref{lem:mod-by-alg} applied to $\id_A:A \to A$,
and that there is also a $B$-module structure on $A$
by the same lemma applied to $f:B \to A$.
Let us denote by ${}_A A$ the former $A$-module
and by ${}_B A$ the latter $B$-module.
Then we have $f^* {}_B A \simeq {}_A A$ as $A$-modules.
\end{rmk}

Next we consider another functor 
\[
 f_!: \, \cMod{\opP}{B} \longto \cMod{\opP}{A}
\]
given by the following coequalizer.
\[
 \xymatrix{
  U_{\opP}(A) \otimes U_{\opP}(B) \otimes M 
  \ar@<0.5ex>[rr]^(0.55){U_{\opP}(f) \otimes \id} 
  \ar@<-0.5ex>[rr]_(0.55){\id \otimes \lambda} & &
  U_{\opP}(A) \otimes M \ar[r] & f_!(M).
 }
\]
Here $\lambda:U_{\opP}(B) \otimes M \to M$ denotes 
the left $U_{\opP}(B)$-module structure on $M$.
The functor $f_!$ is in fact a left adjoint of 
the restriction functor $f^*$.

\begin{fct}[{\cite[Proposition 12.3.10]{LV}}]\label{fct:f_!-f^*}
For any operad $\opP$ in $\catC$ and any morphism
$f:B \to A$ of $\opP$-algebras in $\catE$,
we have an adjoint pair of functors
\[
 f_!: \, \cMod{\opP}{B} \longbij \cMod{\opP}{A} \, :f^*.
\]
In particular we have a natural isomorphism
\[
 \cMod{\opP}{A}(f_!(M),N) \simeq \cMod{\opP}{B}(M,f^*(N))
\]
for each $M \in \cMod{\opP}{B}$ and $N \in \cMod{\opP}{A}$.
\end{fct}

\subsection{Derivations}

Now we recall derivations of algebras over operads
following \cite[\S4.4]{Fr} and \cite[\S12.3.7]{LV}.

\begin{dfn}\label{dfn:derivation}
Let $f:B \to A$ be a morphism 
in the category ${}_{\opP}\catE$ of $\opP$-algebras in $\catE$.
Let $E=(E,\lambda_E)$ be an $A$-module in $\catE$.
A morphism $\theta \in \catE(B,E)$ is called an \emph{$A$-derivation}
if it makes the following diagram commutative.
\[
 \xymatrix{
  S(\opP,B) \ar[d]_{\mu_B} \ar@{=}[r] &
  S(\opP,B;B) \ar[rr]^{S(\id,\id;\theta)} & & 
  S(\opP,B;E) \ar[rr]^{S(\id,f;\id)} & &
  S(\opP,A;E) \ar[d]^{\lambda_E} \\
  B \ar[rrrrr]_{\theta} & & & & & E
 }
\]

We denote by $\Der^{\opP}_A(B,E)$ the set of all such derivations.
\end{dfn}

Using the partial composition $\circ_i$ in \eqref{eq:circ_i},
the condition for $\theta$ is restated as 
$\theta \circ p = \sum_{i=1}^n p \circ_i \theta: 
 B^{\otimes n} \longto E$
for all $n \in \bbN$ and $p \in \opP(n)$.
Using this presentation one can recover the classical notion of derivations 
in the case $\opP=\opCom,\opAsc,\opLie$.

Next we will discuss the functor given by the operation 
$B \mapsto \Der^{\opP}_A(B,E)$.
The source category of this functor is set to be the comma category 
${}_{\opP}\catE/A$ of objects over $A$ in ${}_{\opP}\catE$.
A direct definition is 

\begin{dfn*}
Let $A \in {}_{\opP}\catE$.
Consider the category whose objects are 
morphisms $f:B \to A$ in ${}_{\opP}\catE$,
and whose morphisms are the commutative diagrams
\[
 \xymatrix{ 
  B \ar[rr] \ar[rd]_{f} & & B' \ar[ld]^{f'} \\ & A}
\]
We call it
the \emph{category of $\opP$-algebras over $A$},
and denote it by ${}_{\opP}\catE/A$. 
\end{dfn*}

The functor 
\[
 \Der^{\opP}_{A}(-,E): \, {}_{\opP}\catE/A \longto \cSet
\]
is representable by 
the \emph{module of K\"{a}hler differentials}.
Its definition is as follows.
Recall the free $A$-module functor $A \otimes^{\opP} -$ 
in Definition \ref{dfn:free-A-mod}.
Let $A$ be a $\opP$-algebra in $\catE$. 
We define $\Omega_{\opP} A \in \catE$ to be the following coequalizer.
\[
 \xymatrix{
  A \otimes^{\opP} S(\opP,A) 
  \ar@<0.5ex>[rr]^(0.55){\wt{\mu}_{(1)}} 
  \ar@<-0.5ex>[rr]_(0.55){\wt{\mu}_A} & &
  A \otimes^{\opP} A \ar[r] & \Omega_{\opP} A
 },
\]
where the map $\wt{\mu}_{(1)}$ comes from 
\[
 S(\opP,A;S(\opP,A)) = 
 S(\opP,A;S(\opP,A;A)) \to S(\opP \circ_{(1)} \opP,A;A)
 \xrightarrow{S(\mu,\id;\id)} S(\opP,A;A)
\]
and the map $\wt{\mu}_{A}$ is comes from 
\[
 S(\opP,A;S(\opP,A)) \xrightarrow{S(\id,\id;\mu_A)} S(\opP,A;A).
\]
The $A$-module structure on $A \otimes^{\opP} A$ passes to 
the quotient $\Omega_{\opP}A$.

\begin{dfn}\label{dfn:kahler}
The $A$-module $\Omega_{\opP}A$ in $\catE$ 
is called \emph{the module of K\"{a}hler differentials}. 
\end{dfn}

Now let $f:B \to A$ be a morphism in ${}_{\opP}\catE$.
Recall the functor $f_!:\cMod{\opP}{B} \to \cMod{\opP}{A}$
(see Fact \ref{fct:f_!-f^*}).

\begin{fct*}[{\cite[\S4.4]{Fr}, \cite[\S12.3]{LV}}]
For any $E \in \cMod{\opP}{A}$ we have 
\[
 \cMod{\opP}{A}(\Omega_{\opP}B, f_! E) \deq \Der^{\opP}_A(B,E).
\]
\end{fct*}

\subsection{Cooperads}

Cooperad is a natural dual notion of operad.
See \cite[\S5.8]{LV} for an explanation in the non-relative setting.
Here we give definitions in the relative setting.
Let $\catC$ and $\catE$ be as in the previous subsections.

First we introduce a new monoidal operation on the category $\catC^{\sg}$
of $\sg_*$-objects,
which can be considered as a natural dual operation of $\circ$ 
in Definition \ref{dfn:circ}.

\begin{dfn*}
For $M,N \in \catC^{\sg}$, define $M \bcirc N \in \catC^{\sg}$ by
\[
 M \bcirc N := \prod_{k \ge 0} (M(k) \otimes N^{\otimes k})^{\sg_k}.
\]
Here $\sg_k$ acts diagonally on $M(k) \otimes N^{\otimes k}$
where the action on $N^{\otimes k}$ is the permutation.
\end{dfn*}

One can check that $(\catC^{\sg},\bcirc,\opI)$ is a monoidal category.
Using $\bcirc$ instead of $\circ$,
one can define cooperads just dually as operads.

\begin{dfn*}
A \emph{cooperad in $\catC$} is a $\sg_*$-object $\opC$ in $\catC$ 
together with two morphisms of $\sg_*$-objects 
$\Delta: \opC \to \opC \circ \opC$ and $\ve: \opC \to \opI$
making the following diagrams commutative.
\[
 \xymatrix{
  \opC \ar[r]^{\Delta} \ar[d]_(0.45){\Delta} & 
  \opC \bcirc \opC \ar[d]^(0.45){\id \bcirc \Delta} \\
  \opC \bcirc \opC \ar[r]_(0.42){\Delta \bcirc \id} & 
  \opC \bcirc \opC \bcirc \opC
 }
 \qquad
 \xymatrix{
  & \opC \ar[ld]_{\sim} \ar[d]_{\Delta} \ar[rd]^{\sim} \\
  \opI \bcirc \opC & 
  \opC \bcirc \opC \ar[l]^{\ve \bcirc \id} 
                   \ar[r]_(0.45){\id \bcirc \ve} & 
  \opC \bcirc \opI   
 } 
\]
The morphism $\Delta$ is called the \emph{decomposition map},
and $\ve$ is called the \emph{counit map}.

A morphism between cooperads is defined naturally.
The category of cooperads in $\catC$ is denoted by $\cCo_{\catC}$.
\end{dfn*}

As in the case of the identity operad,
the identity $\sg_*$-object $\opI$ has a cooperad structure.
The corresponding cooperad is called \emph{the identity cooperad} 
and denoted by the same symbol $\opI$.

One can also define coalgebras over cooperads 
in a dual manner of algebras over operads
For $C \in \catC^{\sg}$ and $X \in \catE$, set
\[
 \ol{S}(C,X) := 
 \prod_{k \ge 0} (C(k) \otimes X^{\otimes k})^{\sg_k},
\]
where $\sg_k$ acts diagonally on $C(k) \otimes X^{\otimes k}$.

\begin{dfn*}
Let $\opC$ be a cooperad in $\catC$.
A \emph{$\opC$-coalgebra $C$} in $\catE$ is an object $C \in \catE$
equipped with a morphism $\Delta_C: C \to \ol{S}(\opC,C)$
such that the following diagrams commute.
\begin{align*}
&\xymatrix{
  \opC \ar[rrr]^{\Delta_C} \ar[d]_{\Delta_C} & & &
  \ol{S}(\opC,C) \ar[d]^{\ol{S}(\Delta,\id)} \\
  \ol{S}(\opC,C) \ar[rr]_(0.45){\ol{S}(\id,\Delta_C)} & & 
  \ol{S}(\opC,\ol{S}(\opC,C)) \ar[r]_{\sim} &
  \ol{S}(\opC \bcirc \opC,A)
 } 
 \\ 
& \hskip 7em
 \xymatrix{
  C \ar[d]_{\Delta_C} \ar[rd]^{\sim} \\
  \ol{S}(\opC,C) \ar[r]_{\ol{S}(\eta,\id)} & \ol{S}(\opI,C)
 }
\end{align*}
\end{dfn*}

The structure of a cooperad is encoded by the \emph{partial decompositions} 
\[
 \Delta_{(i)}: \opC(r+s-1) \longto \opC(r) \otimes \opC(s)
\]
for $1 \le i \le r+s-1$ 
as in the case of operads and partial compositions.

Dualizing the argument of free operad, 
we have the notion of cofree cooperad.
Let us briefly recall its definition.
See \cite[\S5.8.6]{LV} for the detail.

In the adjoint language, the cofree cooperad is defined as the image of 
a given object $M \in \catC$ by the left adjoint 
\[
 \ol{\opF}: \catC^{\sg} \longto \cCo_{\catC}
\]
of the forgetful functor $\cCo_{\catC} \to \catC^{\sg}$.

Explicitly, the underlying $\sg_*$-object is given by
the same one as the free operad $\opF(M)$.
In other words, we set 
\[
 \ol{\opF}(M) \seteq \varinjlim \ol{\opF}_n(M),
\]
where $\ol{\opF}_n(M)$ is defined by 
\[
 \ol{\opF}_0(M) \seteq \opI, \quad
 \ol{\opF}_n(M) \seteq \opI \oplus \bigl(M \bcirc \ol{\opF}_{n-1}(M)\bigr).
\]

The decomposition map is inductively defined in the following way.
First we set
\[
 \Delta(\id) \seteq \id \bcirc \id, \quad
 \Delta(m)   \seteq \id \bcirc m + m \bcirc \id^{\otimes n} \ 
 (m \in M(n)).
\]
Thus $\Delta$ on $\ol{\opF}_1(M)$ is defined.
Suppose $\Delta$ is defined on $\ol{\opF}_{n-1}(M)$.
Then for 
$p
 \in M \bcirc \ol{\opF}_{n-1}(M) \subset \ol{\opF}_n(M)$ 
we set
\[
 \Delta(p) \seteq \id \bcirc p + \Delta^+(p),
\]
where $\Delta^+$ is the following composition.
\begin{align*}
 M \bcirc \ol{\opF}_{n-1}M 
&\xrightarrow{\, \id \bcirc \Delta \, }
 M \bcirc (\ol{\opF}_{n-1}M \bcirc \ol{\opF}_{n-1}M )
 \longsimto (M \bcirc \ol{\opF}_{n-1}M) \bcirc \ol{\opF}_{n-1}M \\
&\xrightarrow{\, j_n \bcirc \, i_n \, }
 \ol{\opF}_n M \bcirc \ol{\opF}_n M.
\end{align*}
Here $i_n: \ol{F}_{n-1}(M) \inj \ol{F}_{n-1}(M)$ and 
$j_n: M \bcirc \ol{F}_{n-1}(M) \inj \ol{F}_n(M)$ are natural inclusions.

Define $\ve:\ol{\opF}(M) \to M$ by $\ol{\opF}_1(M) = \opI \oplus M \surj M$.
Then $(\ol{\opF}(M),\Delta,\ve)$ is a cooperad in $\catC$.

\begin{dfn}\label{dfn:cofree-coop}
The cooperad $\ol{\opF}(M)$ is called the \emph{cofree cooperad} 
associated to $M$.
\end{dfn}

Moreover, the inclusion map $\eta: \opI \to \ol{\opF}(M)$ induced by 
by the natural inclusions $j_n$ is a \emph{coaugmentation} map 
of the cooperad $\ol{\opF}(M)$.
In other words, $\eta$ is a cooperad morphism such that $\ve \eta = \id$.

\begin{dfn}\label{dfn:coaug-coop}
A cooperad equipped with a coaugmentation map is called
a \emph{coaugmented cooperad}.
For a coaugmented cooperad $\opC$ with the coaugmentation map 
$\eta: \opI \to \opC$,
the cokernel $\Coker(\eta)$ is called the \emph{coaugmentation coideal},
denoted by $\opC_{\coaug}$.
\end{dfn}

\section{Operadic homology algebra}
\label{sect:op-hom-alg}

In this section we review the homology theory for algebras over operads
following \cite[Chap.\ 13]{Fr} and \cite[Chap.\ 12]{LV}.
Unless otherwise stated, we take $\catC=\catE=\dgkMod$, 
the category of complexes of $\bbk$-modules, or dg $\bbk$-modules, 
over a fixed commutative ring $\bbk$.
An object of $\dgkMod$ will be denoted like 
$C=(\oplus_{n\in\bbZ}C_n,d)=(C_{\bullet},d)$ with $d:C_n \to C_{n+1}$.
For a homogeneous element $x \in C_p$, we set $|x|:=p$.
We consider the standard symmetric monoidal structure on $\dgkMod$.
Explicitly, for $(C_{\bullet},d_C),(D_{\bullet},d_D) \in \dgkMod$ 
we define $C \otimes D=((C\otimes D)_\bullet, d)$ by 
\[
 (C \otimes D)_n := \bigoplus_{p+q=n}C_p \otimes D_q,
\]
and 
\[
 d(x \otimes y) := d_C(x) \otimes y + (-1)^{p} x \otimes d_D(y) 
\]
for a homogeneous element $x \in C_p$.

\subsection{Homology of algebras over operads}

Following \cite[Chapter 13]{Fr} we recall 
the homology theory of algebras over operads.
Since we take $\catC=\dgkMod$, 
it has a standard model structure, 
making $\catC$ to be a cofibrantly generated symmetric monoidal category.
Let $\opP$ be a $\sg_*$-cofibrant operad in $\catC$.
Then by the argument on semi-model structure on operads \cite[Chap.\ 11]{Fr}
the category ${}_{\opP}\catE$ of $\opP$-algebras in $\catE=\dgkMod$ 
has a semi-model structure.
In particular for any $A \in {}_{\opP}\catE$ 
one can take a cofibrant replacement $Q_A \simto A$ in ${}_{\opP}\catE$. 

For the definition of (co)homology of algebras over operads,
it is convenient to introduce the universal coefficient.
Recall the set $\Omega_{\opP}(\cdot)$ of 
K\"{a}hler differentials in Definition \ref{dfn:kahler}.
Since we take $\catC=\catE=\dgkMod$,
$\Omega_{\opP}(M)$ is actually a dg $\bbk$-module 
for any $M \in \cMod{\opP}{A}$
In particular, taking a cofibrant replacement $Q_A \simto A$,
one has the dg $\bbk$-module 
$\Omega_{\opP}(Q_A) \in \cMod{\opP}{A}$.
Note that by Fact \ref{fct:rep=mod}
$\Omega_{\opP}(Q_A)$ is a left $U_{\opP}(Q_A)$-module,
where $U_{\opP}(Q_A)$ denotes the enveloping algebra of $Q_A$.
(see Definition \ref{dfn:env:op-alg}).
Note also that by the morphism $Q_A \to A$
the enveloping algebra $U_{\opP}(Q_A)$ acts on $U_{\opP}(A)$ from right. 

\begin{dfn*}
For $A \in {}_{\opP}\catE$ 
we define the left $U_{\opP}(A)$-module $T_{\opP}(Q_A)$ to be
\[
 T_{\opP}(Q_A) \seteq 
 U_{\opP}(A) \otimes_{U_{\opP}(Q_A)} \Omega_{\opP}(Q_A).
\]
\end{dfn*}

Let us cite the following important result.

\begin{fct}[{\cite[13.1.12 Lemma]{Fr}}]\label{fct:T^1:weq}
The morphism $T_{\opP}(f):T_{\opP}(Q_A) \to T_{\opP}(Q_B)$ induced by
a weak equivalence $f:Q_A \simto Q_B$ of $\opP$-algebras over $A$ in $\catE$
is a weak equivalence of left $U_{\opP}(A)$-modules
if $Q_A$ and $Q_B$ are cofibrant.
\end{fct}

Now we can define the (co)homology of $\opP$-algebras with coefficients
using $\otimes$ and $\Hom$ over the modules of the associative algebra 
$U_{\opP}(Q_A)$.

\begin{dfn}\label{dfn:homology}
Let $A \in {}_{\opP}\catE$ 
and $Q_A$ be a cofibrant replacement of $A$ in $\catE$.
\begin{enumerate}
\item
For a right $U_{\opP}(A)$-module $M$, 
the \emph{homology of $A$ with coefficient in $M$} is 
a dg $\bbk$-module defined by 
\[
 H^{\opP}_{\bullet}(A,M) \seteq 
 H_{\bullet}\bigl(M \otimes_{U_{\opP}(A)}T_{\opP}(Q_A)\bigr),
\]

\item
For a left $U_{\opP}(A)$-module $N$, 
the \emph{cohomology of $A$ with coefficient in $N$} is 
a graded $\bbk$-module defined by 
\[
 H_{\opP}^{\bullet}(A,N) \seteq H^{\bullet} \Hom_{U_{\opP}(A)}(T_{\opP}(Q_A),N).
\]
\end{enumerate}
\end{dfn}

By Fact \ref{fct:T^1:weq},
these definitions are 
independent of the choice of cofibrant replacement $Q_A$
\cite[13.1.2, 13.1.4 Proposition]{Fr}.


Let us close this subsection by explaining 
the origin of the definition of homology.
Recall the set $\Der^{\opP}_A(B,E)$ of derivations 
in Definition \ref{dfn:derivation}.
In the present setting $\catC=\catE=\dgkMod$,
this set is a dg $\bbk$-module.

\begin{fct*}[{\cite[Chap.\ 13]{Fr}}]
For $A \in {}_{\opP}\catE$ and $N \in \cMod{\opP}{A}$, 
we have the following dg $\bbk$-module isomorphism. 
\[
 H_{\opP}^{\bullet}(A,N) \simeq H^{\bullet} \Der_A^{\opP}(Q_A,N).
\]
\end{fct*}

In \cite[Chap.\ 12]{LV} the right hand side is taken to be 
the definition of the homology of $A$ with coefficient.
It is reminiscent of the Andr\'e-Quillen homology of commutative algebras.

%
%

\subsection{Operadic twisting morphism}

The remaining subsections in the present section
are devoted to the explanation of
operadic cochain complexes for Koszul operads.
The main subject will be given in \S\ref{subsec:occ}.

We give a relative version of operadic twisting morphism 
explained in \cite[\S6.4]{LV}.
In this subsection $\catC$ and $\catE$ are taken generally 
as in \S\ref{sect:operad}.
%
For an operad $(\opP,\mu_{\opP},\eta)$ and 
a cooperad $(\opC,\Delta_{\opC},\ve)$ in the base category $\catC$, consider
\[
 \catC(\opC,\opP) \seteq \{\catC(\opC(n),\opP(n))\}_{n\ge0},
\]
which is a $\sg_*$-object under the action 
\[
 (f \cdot \sigma)(x) \seteq f(x \cdot \sigma^{-1})\cdot \sigma
\]
for $f \in \catD(\opC(n),\opP(n))$, $x \in \opC(n)$ and $\sigma \in \sg_n$.

Now we define $\mu(f;g_1,\ldots,g_k;\sigma) \in \catC(\opC(n),\opP(n))$ 
for $\sigma \in \sg_n$, $f \in \catC(\opC(k),\opP(k))$ and 
$g_j \in \catC(\opC(i_j),\opP(i_j)$, $1 \le j \le k$
with $i_1+\cdots+i_k=n$ to be the following composition of morphisms.
\begin{align*}
 \opC(n) & \xrightarrow{\ \Delta_{\opC} \ } 
 (\opC \bcirc \opC)(n) \longsurj 
 \opC(k) \otimes \opC(i_1) \otimes \cdots \otimes \opC(i_k) \otimes \sg_n \\
&\xrightarrow{f \otimes g_1 \otimes \cdots \otimes g_k \otimes \sigma}
 \opP(k) \otimes \opP(i_1) \otimes \cdots \otimes \opP(i_k) \otimes \sg_n
 \longto (\opP \circ \opP)(n)
 \xrightarrow{\ \mu_{\opP} \ } \opP(n)
\end{align*}
Then we have

\begin{fct*}[{\cite[Propsoition 6.4.1]{LV}}]
$\catC(\opC,\opP)$ with $\mu$ the composition map is an operad in $\catC$.
\end{fct*}

Following \cite[\S6.4]{LV} we use 

\begin{dfn*}
The operad $\catC(\opC,\opP)$ is called the \emph{convolution operad}.
\end{dfn*}

Now we change the base category to the dg category $\catD$ 
explained in the beginning of this section.
Each object $X$ in $\catD$ has a $\bbZ$-grading $X_{\bullet}$ 
and a differential $d: X_{\bullet} \to X_{\bullet-1}$.
A $\sg_*$-object $M=\{M(n)\}_{n\ge0}$ in $\catD$
consists of objects $M(n)$ in $\catD$ with $\sg_n$-action.

\begin{dfn*}
For a homogeneous morphism $f:M \to N$ in $\catD^{\sg}$ of degree $|f|$,
define its derivative $\partial(f)$ to be 
\[
 \partial(f) \seteq d_N \circ f - (-1)^{|f|}f \circ d_M.
\]
\end{dfn*}

The importance of the convolution operad lies in 

\begin{fct}[{\cite[Proposition 6.4.5]{LV}}]\label{fct:convol:dgla}
The convolution operad $(\catD(\opC,\opP),\partial)$ 
is an operad in the dg category $\catD$.
It is also equipped with a pre-Lie product 
\[
 f \star g: \opC \xrightarrow{ \Delta_{(1)} } 
 \opC \circ_{(1)} \opC \xrightarrow{f \circ_{(1)} g} 
 \opP \circ_{(1)} \opP \xrightarrow{\mu_{(1)}} \opP
\]
for $f,g \in \Hom(\opC,\opP)$,
so that 
$(\catD(\opC,\opP),[\,,\,],\partial)$ with 
$[f,g]:=f \star g - g \star f$ is a dg Lie algebra.
\end{fct}

For the definition of pre-Lie algebra, see \cite[\S1.4]{LV}.

A solution $\alpha \in \catD(\opC,\opP)$ of the Maurer-Cartan equation
\[
 \partial(\alpha) + \alpha \star \alpha = 0
\]
of degree $-1$ is called an \emph{operadic twisting morphism}.
For such $\alpha$, one can define a complex
\begin{align*}
 \opP \circ_{\alpha} \opC = 
 (\opP \circ \opC,d_{\opP \circ \opC}+d^l_{\alpha})
\end{align*}
called the \emph{left twisting composite product},
and another one 
\begin{align}\label{eq:ltcp}
 \opC \circ_{\alpha} \opP = 
 (\opC \circ \opP,d_{\opC \circ \opP}+d^r_{\alpha})
\end{align}
called the \emph{right twisting composite product}.
See \cite[\S6.4.5]{LV} for the detail.

\subsection{Operadic bar and cobar construction}

We explain a relative version of 
the operadic bar and cobar constructions given in \cite[\S6.5]{LV}.

Let $\catD$ be a symmetric monoidal dg category linear over a field $\bbk$.
The operadic bar construction means the following functor
\[
 \opBar: \ \{\text{augmented operads in $\catD$}\} \longto 
           \{\text{coaugmented cooperads in $\catD$}\}. 
\]

%

Let us explain the definition of $\opBar$.
Let $\bbk s$ be the one-dimensional graded vector space spanned by $s$
with grading $|s|=1$.
The suspension of $V \in \catD$ is defined to be
\[
 s V := \bbk s \otimes V.
\]
In particular we have $(S V)_i \simeq V_{i-1}$.

For an augmented operad $\opP=(\opP,\mu_{\opP},\eta,\ve)$ in $\catD$,
we set
\[
 \opBar(\opP) := \bigl(\ol{\opF}(s \opP_{\aug}),d_1+d_2\bigr).
\]
Here $\ol{\opF}$ denotes the cofree cooperad functor
explained around Definition \ref{dfn:cofree-coop},
and we applied it to the suspended $\sg_*$-object $s \opP_{\aug}$,
where $\opP_{\aug}$ is the augmentation ideal
(see Definition \ref{dfn:augmented-operad}).

The differentials $d_1$ and $d_2$ are defined as follows.
$d_1$ is the map induced on $\ol{F}(s\opP_{\aug})$ from 
the differential $d_{\opP}$ on $\opP$.
Thus we have $d_1^2=0$.
$d_2$ is given by the following composite.
\begin{align*}
 \ol{\opF}(s \opP_{\aug}) 
&\longsurj
 s \opP_{\aug} \circ_{(1)} s \opP_{\aug} \longsimto
 (\bbk s \otimes \opP_{\aug}) \otimes (\bbk s \otimes \opP_{\aug})
\\
&\xrightarrow{\, \id \otimes \tau \otimes \id \,}
 (\bbk s \otimes \bbk s) \otimes (\opP_{\aug} \otimes \opP_{\aug})
 \xrightarrow{\, \mu_s \otimes \mu_{\opP} \,}
 \bbk s \otimes \opP_{\aug}.
\end{align*}
Here $\tau$ is the permutation map,
and $\mu_s: \bbk s \otimes \bbk s \to \bbk s$ is the map of degree $-1$
defined by $\mu_s(s \otimes s) = s$.
By \cite[Proposition 6.5.1]{LV} we know $d_2^2=0$.
We also have $d_1d_2+d_2d_1=0$,
so that $\opBar(\opP)$ is a cooperad in $\catD$.

The cooperad $\opBar(\opP)$ obtained is coaugmented
by the coaugmentation on the cofree cooperad.
In summary, the functor $\opBar$ is well-defined.

Dually one can construct a functor
\[
 \opCob: \ \{\text{coaugmented cooperads in $\catD$}\} \longto 
           \{\text{augmented operads in $\catD$}\}.
\]
For an coaugmented cooperad $\opC=(\opC,\Delta_{\opC},\ve,\eta)$ in $\catD$,
we set
\[
 \opCob(\opC) := \bigl(\opF(s^{-1} \opC_{\coaug}),d_1+d_2\bigr).
\]
Here $\opF$ denotes the free operad functor,
and $s^{-1}$ denotes the inverse suspension, 
i.e., 
the graded vector space $\bbk s^{-1}$ generated by $s^{-1}$ with $|s^{-1}|=-1$.
$\opC_{\coaug}$ denotes the coaugmentation coideal of $\opC$,
see Definition \ref{dfn:coaug-coop}.

The differential $d_1$ is induced by the one $d_{\opC}$ on the cooperad $\opC$.
The differential $d_2$ is defined by
\begin{align*}
 \bbk s^{-1} \otimes \opC_{\coaug} 
&\xrightarrow{ \, \Delta_s \otimes \Delta_{\opC} \, }
 (\bbk s^{-1} \otimes \bbk s^{-1}) \otimes (\opC_{\coaug} \bcirc_{(1)} \opC_{\coaug})
\\
&\xrightarrow{ \, \id \otimes \tau \otimes \id \, }
 (s^{-1} \otimes \opC_{\coaug}) \bcirc_{(1)} (\bbk s^{-1} \otimes \opC_{\coaug})
\\
&\longsimto
 s^{-1} \opC_{\coaug} \bcirc_{(1)} s^{-1} \opC_{\coaug}
 \longinj
 \opF(s^{-1} \opC_{\coaug}).
\end{align*}

Now we can state the fundamental result 
on the operadic bar and cobar constructions.

\begin{fct*}[{\cite[Theorem 6.5.7]{LV}}]
The functors $\opBar$ and $\opCob$ form an adjoint pair
\[
 \opCob: \{\text{coaugmented cooperads in $\catD$}\} \longbij
           \{\text{augmented operads in $\catD$}\}   :\opBar.
\]
More precisely, 
for an augmented operad $\opP$ and a coaugmented cooperad $\opC$ 
in $\catD$, there exists a natural isomorphism
\[
 \cOp_{\catD}(\opCob(\opC),\opP) \simeq 
 \cCo_{\catD}(\opC,\opBar(\opP)).
\]
\end{fct*}

\begin{fct*}[{\cite[Theorem 6.6.3]{LV}}]
The unit $\eta: \opBar \opCob \opC \simto \opC$ 
is a quasi-isomorphism of cooperads in $\catD$.
Dually, the counit $\ve: \opCob \opBar \opP \simto \opP$ 
is a quasi-isomorphism of operads in $\catD$.
\end{fct*}

\subsection{Operadic (co)chain complexes}
\label{subsec:occ}

Following \cite[\S12.1]{LV}, we recall a general construction of 
complexes computing the (co)homology of algebras over quadratic operads.
We take $\catC=\catE=\dgkMod$ as in the beginning of this section.

For a graded $\sg_*$-module $M$ in $\catC$,
consider the free operad  $\opF(M)$ in $\catC$.
It has a \emph{weight grading} given by 
\begin{align*}
&w(\id):= 0, \quad w(\xi) ;= 1 \ (\xi \in M(n)),
\\
&w(\xi;\eta_1,\ldots,\eta_k;\sigma):=w(\xi)+w(\eta_1)+\cdots+w(\eta_k).
\end{align*}
Here we used the set-theoretic notation \eqref{eq:compos:elem}
and $\id \in \one = \bbk \subset \opF(M)(1)$.
We denote by $\opF(M)^{(m)}$ the submodule of weight $m$.

Now an operadic \emph{quadratic data}
is a pair $(E,R)$ of a graded $\sg_*$-module $E$ 
and a graded sub $\sg_*$-module $R \subset \opF(E)^{(2)}$.

Associated to such $(E,R)$, we define the \emph{quadratic operad} to be 
\[
 \opP(E,R) := \opF(E)/(R),
\]
the quotient operad of $\opF(E)$ by the operadic ideal generated by $R$.
Thus it is a universal operad among the quotient operads 
$\opF(E) \surj \opP$ such that 
the composition $(R) \inj \opF(E) \surj \opP$ is trivial.
An operad which is isomorphic to such $\opP(E,R)$
is called quadratic.
The classical operads in Example \ref{eg:classical-operads} are thus 
quadratic.

Dually we define the \emph{quadratic cooperad} $\opC(E,R)$ to be
the cooperad which is universal among sub-cooperads 
$\opC \inj \ol{\opF}(E)$ such that the composition 
$\opC \inj \ol{\opF}(E) \surj \ol{\opF}(E)^{(2)}$
is trivial.

\begin{dfn*}
The \emph{Koszul dual cooperad} of the quadratic operad $\opP(E,R)$ 
is defined to be the quadratic cooperad 
\[
 \opC(s E,s^2 R).
\]
Here $s$ denotes the degree shifting in $\catC=\dgkMod$.
\end{dfn*}

Note that $\opP(E,R)$ and $\opC(s E,S^2 R)$ are both 
$\sg_*$-modules in $\catC=\dgkMod$ with trivial differentials.
We now have 

\begin{fct*}[{\cite[Lemma 7.4.1]{LV}}]
Define a morphism $\kappa$ of $\sg_*$-modules in $\catC$ to be
\[
 \kappa: \opC(s E, s^2 R) \longsurj s E \xrightarrow{s^{-1}} E
 \longinj \opP(E,R).
\]
Then $\kappa \star \kappa =0$, where $\star$ is the pre-Lie product 
in Fact \ref{fct:convol:dgla}.
\end{fct*}

Thus $\kappa$ is an operadic twisting morphism.
Now recall the construction of the left twisting composite product 
\eqref{eq:ltcp}.

\begin{dfn*}
Let $\opP=\opP(E,R)$ be a quadratic operad
and $\opC= \opC(s E,s^2 R)$ be the associated Koszul dual cooperad.
\begin{enumerate}
\item
The complex $\opC  \circ_{\kappa} \opP$ is called 
the \emph{Koszul complex} of the operad $\opP$.
\item
$\opP$ is called a \emph{Koszul operad} 
if the Koszul complex $\opC  \circ_{\kappa} \opP$ is acyclic.
\end{enumerate}
\end{dfn*}

The bar-cobar construction explained in the previous subsection implies

\begin{fct*}[{\cite[Theorem 7.4.2]{LV}}]
For a quadratic cooperad $\opP=\opP(E,R)$, the followings are equivalent.
\begin{enumerate}
\item 
$\opP$ is Koszul.
\item
The right twisting composite product 
$\opP  \circ_{\kappa} \opC$ with $\opC=\opC(s E, s^2 R)$ is acyclic.
\item
The canonical inclusion $\opC \inj \opBar(\opC)$ 
induced by the adjunction $(\opCob,\opBar)$ is a quasi-isomorphism.
\item
The canonical projection $\opC \inj \opCob(\opC)$ is a quasi-isomorphism.
\end{enumerate}
\end{fct*}

Now we can given the definition of operadic chain complex.

\begin{dfn*}
Let $\opP=\opP(E,R)$ be a quadratic operad and 
$\opC=\opC(s E,s^2 R)$ be its Koszul dual cooperad.
For a $\opP$-algebra $A$, we set
\[
 C^{\opP}_{\bullet}(A) := (\opC \circ_{\kappa}\opP)\circ_{\opP}A.
\]
\end{dfn*}

Here $\circ_{\opP}$ denotes the relative composite product,
which is defined to be the coequalizer
\[
 \xymatrix{
 M \circ \opP \circ N
 \ar@<0.5ex>[rr]^{\id \circ \lambda} \ar@<-0.5ex>[rr]_{\rho \circ \id} 
 & & M \circ N \ar[r] & M \circ_{\opP} N
 }
\]
for a right $\opP$-module $M$ with the action map 
$\rho:M \circ \opP \to M$
and a left $\opP$-module $N$ with the action map 
$\lambda: \opP \circ N \to N$.
Thus as a $\sg_*$-module we have 
$C^{\opP}_{\bullet}(A) \simeq \opC \circ A = \ol{S}(\opC,A)$,
and we can present
\[
 C^{\opP}_{\bullet}(A)  = (\ol{S}(\opC,A), d).
\]

For a binary quadratic operad $\opP$, 
we have the following explicit description of this complex.

\begin{fct}[{\cite[Proposition 12.1.1]{LV}}]\label{fct:occ=ccc}
Let $\opP=\opP(E,R)$ be a binary quadratic operad,
i.e., $E=(0,0,E,0,\ldots)$ as a $\sg_*$-module.
Then for a $\opP$-algebra $A$, we have 
\[
 C^{\opP}_{n}(A) = \opC(n+1) \otimes_{\sg_{n+1}}A^{\otimes (n+1)}
\]
with the differential
\begin{align*}
&d(\zeta \otimes (a_1,\ldots,a_{n+1}))
\\
&= \sum \xi \otimes (a_{\sigma^{-1}(1)},\ldots,a_{\sigma^{-1}(i-1)},
  \eta(a_{\sigma^{-1}(i)},a_{\sigma^{-1}(i+1)}),
  a_{\sigma^{-1}(i+2)},\ldots,a_{\sigma^{-1}(n+1)})
\end{align*}
for 
$\Delta_{(1)}\zeta = \sum(\xi;\id,\ldots,\id,\eta,\id,\ldots,\id;\sigma)$,
$\zeta \in \opC(n+1)$, $\xi \in \opC(n)$, 
$\eta \in \opC(2)=E$ and $\sigma \in \sg_{n+1}$.
\end{fct}

Similarly, the operadic cochain complex is defined as follows.

\begin{dfn*}
Let $\opP=\opP(E,R)$ be a quadratic operad and 
$\opC=\opC(s E,s^2 R)$ be its Koszul dual cooperad.
For a $\opP$-algebra $A$ and an $A$-module $M$, we set
\[
 C_{\opP}^{\bullet}(A,M) := (\Hom(\opC \circ A,M),\partial_{\kappa}).
\]
where the differential is given by 
\[
 \partial_{\kappa}(g) := \partial(g)-(-1)^{|g|}d.
\]
Here $d$ is the differential of $C^{\opP}_{\bullet}(A)$,
and $\partial$ is given by
\[
 \ol{S}(\opC,A) \xrightarrow{\Delta} \ol{S}(\opC \circ \opC,A)
 \xrightarrow{\ol{S}(\kappa,\wt{\ve};g)}\ol{S}(\opC,A;M)
 \xrightarrow{\lambda_M} M.
\]
The map $\wt{\ve}=\ol{S}(\ve,\id): \ol{S}(\opC,A) \to \opI \circ A \simeq A$
is induced by the counit $\ve:\opC \surj \opI$.
\end{dfn*}

The comparison to Definition \ref{dfn:homology} is given by 

\begin{fct*}[{\cite[Theorem 12.4.3]{LV}}]
Let $\opP=\opP(E,R)$ be a quadratic operad in $\catC$,
$A$ be a $\opP$-algebra and $M$ be an $A$-module in $\catE$.
If the operad $\opP$ is Koszul and the complexes $A,M$ are bounded below,
then the cohomology  $H_{\opP}^\bullet(A,M)$ 
in Definition \ref{dfn:homology} is calculated 
by the complex $C_{\opP}^{\bullet}(A,M)$.
\end{fct*}

\begin{eg}
Using Fact \ref{fct:occ=ccc},
one can check that 
for algebras over the classical operads 
in Example \ref{eg:classical-operads},
the operadic (co)chain complex 
$C^{\opP}_\bullet(A)$ (or $C_{\opP}^\bullet(A,M)$) 
is essentially the classical complex calculating the (co)homology.
In other words they are
\begin{enumerate}
\item
the Harrison (co)chain complex for $\opP=\opCom$,
\item
the Hochschild (co)chain complex for $\opP=\opAsc$,
\item
the Chevalley-Eilenberg (co)chain complex for $\opP=\opLie$.
\end{enumerate}
\end{eg}

\section{Relative homology of operads}
\label{sect:rel-hom}

\subsection{Cotriple homology}

Let us briefly recall the theory of cotriple homology 
following \cite[Chapter 8]{W}.
We begin with the introduction of simplicial objects 
and the associated complexes.

We denote by $\Delta$ the category of finite ordered sets 
$[n] := \{0<1<\cdots<n\}$ ($n \in \bbN$)
and non-decreasing monotone maps.

\begin{dfn*}
For a category $\catC$,
a \emph{simplicial object} in $\catC$ means a functor 
$C:\Delta^{\op} \to \catC$.
\end{dfn*}

Let $C$ be a simplicial object in $\catC$.
We set $C_n := C([n])$.
Then the structure of simplicial object 
is uniquely determined by $\{C_n\}_{n\in\bbN}$
together with face operations $\partial_i: C_n \to C_{n-1}$ 
and degeneracy operations $\sigma_i: C_n \to C_{n+1}$ 
for $i=0,1,\ldots,n$
satisfying the simplicial identities 
(see \cite[Proposition 8.1.3]{W}).

Let $\catA$ be an abelian category.
For a simplicial object $A$ in $\catA$, 
we set 
\[
 C_n(A) := \begin{cases} A_n & (n\ge0) \\ 0 & (n<0)\end{cases}, \quad
 d:=\sum_{i=0}^n(-1)^i \partial_i : C_n(A) \to C_{n-1}(A), 
\]
where $\partial_i:A_n \to A_{n-1}$ is the face operation of $A$.
Then we have $d^2=0$ so that 
$C_{\bullet}(A)=(\{C_n(A)\}_{n\in\bbZ},d)$ is a chain complex in $\catA$.

\begin{dfn*}
The chain complex $C_\bullet(A)$ is called 
the \emph{unnormalized chain complex} of $A$.
\end{dfn*}

Dually we define a \emph{cosimplicial object} in $\catC$ to be 
a functor $\Delta \to \catC$.
For a cosimplicial object $A$ in an abelian category,
setting $A^n := A([n])$,
we have a cochain complex $C^{\bullet}{A}=(\{C^n(A)\}_{n\in\bbZ},d)$
with $C^n(A) = A^n$ and $d:A^n \to A^{n+1}$.

\begin{dfn*}
The cochain complex $C^\bullet(A)$ is called 
the \emph{unnormalized cochain complex} of $A$.
\end{dfn*}

Next we turn to the notion of cotriple and 
a construction of simplicial object from a given cotriple.

\begin{dfn*}
A \emph{cotriple} $(T, \ve, \delta)$ in a category $\catC$ 
consists of a functor $T:\catC \to \catC$,
natural transformations $\ve:T \to \id_{\catC}$ and $\delta:T \to T T$
such that the following diagrams commute.
\begin{align*}
\xymatrix{
 T C \ar[r]^{\delta_C} \ar[d]_{\delta_C} & T T C \ar[d]^{\delta_{T C}} 
 \\
 T T C \ar[r]_{T \delta_{C}} & T T T C
}
\qquad
\xymatrix{
     & T C \ar[ld]_{\id} \ar[d]^{\delta_{C}} \ar[rd]^{\id} \\
 T C & \ar[l]^{T \ve_C} T T C \ar[r]_{\ve_{T C}} & T C
}
\end{align*}
\end{dfn*}

\begin{fct}[{\cite[8.6.4]{W}}]\label{fct:cotriple:simplicial}
Let $(T, \ve, \delta)$ be  a cotriple in a category $\catC$.
For an object $C \in \catC$,
we have a simplicial object $T_{\Delta}C$ in $\catC$ determined by
\[
 (T_{\Delta} C)_n := T^{n+1} C, \quad 
 \partial_i := T^i \ve T^{n-i}, \quad
 \sigma_i := T^i \delta T^{n-i}.
\]
\end{fct}

Let $(T, \ve, \delta)$ be a cotriple in a category $\catC$, 
and $E: \catC \to \catM$ be a functor to an abelian category $\catM$.
Take an object $A \in \catC$.
Then the image $E(T_{\Delta}A)$
is a simplicial object in $\catM$,
since a simplicial object is defined to be a contravariant functor.
Thus we have the unnormalized chain complex $C_{\bullet}(E(T_{\Delta}A))$. 

\begin{dfn}\label{dfn:cotriple:homology}
The \emph{cotriple homology} $H_n(A,E)$ of $A \in \catC$ 
with coefficients in a functor $E:\catC \to \catM$ is defined to be
the homology of $C_{\bullet}(E(T_{\Delta}A))$:
\[
 H_n(A,E) := H_n C_{\bullet}(E(T_{\Delta}A)).
\]
\end{dfn}

Dually, given a functor $F: \catC^{\op} \to \catM$ to an abelian category $\catM$,
we have a cosimplicial object $F(T_{\Delta}A)$ in $\catM$,
so that we have the unnormalized cochain complex $C^{\bullet}(F(T_{\Delta}A))$. 
It is natural to introduce

\begin{dfn}\label{dfn:cotriple:cohomology}
The \emph{cotriple cohomology} $H^n(A,F)$ of $A \in \catC$ 
with coefficients in a functor $F:\catC^{\op} \to \catM$ is defined to be
the homology of $C^{\bullet}(F(T_{\Delta}A))$:
\[
 H^n(A,F) := H^n C^{\bullet}(F(T_{\Delta}A)).
\]
\end{dfn}

Finally we recall a construct of cotriple from a given
adjoint pair of functors, following \cite[8.6.2]{W}.
Let 
\[
 F: \catB \leftrightarrows \catC :U
\]
be a adjoint pair of functors.
Set 
\[
 T := F U: \, \catC \longto \catC.
\]
The adjunction gives natural transformations 
\[
 \ve: T= F U \to \id_{\catC}, \quad \eta: \id_{\catB} \to U F.
\]
Also we define a natural transformation $\delta:T \to T T$ to be 
\[
 \delta_C := F(\eta_{U C}): \, F(U C) \longto F(U F U C)
\]
for $C \in \catC$.

\begin{fct}\label{fct:cotriple:adj}
The obtained data $(F U, \ve, \delta)$ is a cotriple in $\catC$.
\end{fct}

\subsection{Relative homology of algebras over operads}

Let $\catC$ and $\catE$ be symmetric monoidal categories 
as in \S\ref{sect:operad} 
and $\opP$ be an operad in $\catC$.
Recall that ${}_{\opP}{\catE}$ denotes the category of $\opP$-algebras 
in $\catE$ (see Definition \ref{dfn:P-alg:hom}).


Let us fix $f: B \to A$ in ${}_{\opP}\catE$ for a while.
Recall that by Fact \ref{fct:f_!-f^*} we have an adjoint pair 
\[
 f_!: \, \cMod{\opP}{A} \longbij \cMod{\opP}{B} \, :f^*
\]
of functors,
where $f^*$ is the restriction functor and 
$f_!:=U_{\opP}(A) \otimes_{U_{\opP}(B)} -$. 
Now we apply Fact \ref{fct:cotriple:adj} to this adjoint pair,
and get a cotriple $(f_!f^*,\ve,\delta)$ in $\cMod{\opP}{A}$.
Then applying Fact \ref{fct:cotriple:simplicial} to this cotriple
we have a simplicial object
$\BD{\opP}(A,B,M)$ in $\cMod{\opP}{A}$ for $M \in \cMod{\opP}{A}$ with
\[
 \BD{\opP}(A,B,M)_n = (f_!f^*)^{n+1}M.
\]
As a left $U_{\opP}(A)$-module we have
\[
 \BD{\opP}(A,B,M)_n = 
  (U_{\opP}(A) \otimes_{U_{\opP}(B)})^{n+1} M,
\]
where we omitted the restriction functor $f^*$.
In particular, taking $M=A$,
we have a simplicial left $U_{\opP}(A)$-module $\BD{\opP}(A,B,A)$.

We now assume $\catE=\dgkMod$ and consider the functors 
\[
 M \otimes_{U_{\opP}(A)} -: \, \cMod{\opP}{A} \longto \cMod{\opP}{A}
\]
and 
\[
 \Hom_{U_{\opP}(A)}(-,N): \, 
 \bigl(\cMod{\opP}{A}\bigr)^{\op} \longto \cMod{\opP}{A},
\] 
where $M$ (resp.\ $N$) is a right (resp.\ left) 
$U_{\opP}(A)$-module in $\catE$.
Then applying Definition 
\ref{dfn:cotriple:homology}, \ref{dfn:cotriple:cohomology} to 
these functors and the simplicial object $\BD{\opP}(A,B,A)$ 
in $\cMod{\opP}{A}$,
we obtain the unnormalized (co)chain complex.
It gives the notion of relative (co)homology for algebras over operads.

\begin{dfn}\label{dfn:homology:relative}
Let $f:B \to A$ be a morphism in ${}_{\opP}\catE$ with $\catE=\dgkMod$.
\begin{enumerate}
\item 
For a right $U_{\opP}(A)$-module $M$ in $\catE$, we define 
\emph{the relative homology of $A$ over $B$ with coefficient in $M$} to be 
\[
 H_n(A,B,M) := 
 H_n C_{\bullet}(M \otimes_{U_{\opP}(A)} \BD{\opP}(A,B,A)).
\]

\item 
For a left $U_{\opP}(A)$-module
(or equivalently, an $A$-module) $N$ in $\catE$, we define 
\emph{the relative cohomology of $A$ over $B$ with coefficient in $N$} to be 
\[
 H^n(A,B,N) := 
 H^n C^{\bullet}\Hom_{U_{\opP}(A)}(\BD{\opP}(A,B,A),N).
\]
\end{enumerate}
\end{dfn}

The following statement is a relative version of \cite[13.3.4 Theorem]{Fr},
which claims that the cotriple homology coincides with 
the homology of algebras over operads in Definition \ref{dfn:homology}.

\begin{thm}
Let $\opP$ be a $\sg_*$-cofibrant operad in $\catC=\dgkMod$,
and $A$ be a $\opP$-algebra in $\catE=\dgkMod$.
Then the relative (co)homology in Definition \ref{dfn:homology:relative}
with $B=\bbk$ the trivial $\opP$-algebra
is isomorphic to the (co)homology in Definition \ref{dfn:homology}.
In other words, we have
\[
 H_n(A,M) = 
 H_n C_{\bullet}\bigl(M \otimes_{U_{\opP}(A)}\BD{\opP}(A,\bbk,A)\bigr)
\]
and 
\[
 H^n(A,N) = 
 H^n C^{\bullet}\Hom_{U_{\opP}(A)}\bigl(\BD{\opP}(A,\bbk,A),N\bigr).
\]
\end{thm}

\section{Operadic semi-infinite cohomology}
\label{sect:std}

In this section we introduce the main ingredient of this note.
Let $\bbk$ be a field.
We set 
$\catC=\catE=\dgkMod$.
The monoidal structure will be denote by $\otimes_{\bbk}$
or by $\otimes$ for simplicity. 
We fix an operad $\opP$ in $\catC$.

\subsection{Semi-infinite structure}

Let $A$ be a $\opP$-algebra in  $\catE$.
Recall that $U_{\opP}(A)$ denotes the enveloping algebra of $A$,
which is a dg algebra over $\bbk$ in the present situation.
We denote by $U_{\opP}(A)= \oplus_{n \in \bbZ} U_{\opP}(A)_n$
the grading structure.
Mimicking the notion of semi-infinite structure for associative algebras
in \cite{A2}, we introduce

\begin{dfn}\label{dfn:semi-inf-str}
The \emph{semi-infinite structure} on $A$ is 
the monomorphism
\[
 f: B \longinj A
\]
in $\catE$ such that the following conditions hold.
We set $N := \Coker(f)$.
\begin{enumerate}
\item 
The enveloping algebra $U_{\opP}(A)$ contains 
$U_{\opP}(N)$ and $U_{\opP}(B)$ as graded subalgebras.
\item
$U_{\opP}(N)$ is non-negatively graded,
$U_{\opP}(N)_0=\bbk$ and $\dim_{\bbk} U_{\opP}(N)_n < \infty$.
\item
$U_{\opP}(B)$ is non-positively graded.
\item
The multiplication in $U_{\opP}(A)$ gives isomorphisms of 
graded vector spaces
$U_{\opP}(B) \otimes_{\bbk} U_{\opP}(N) \simto U_{\opP}(A)$ and 
$U_{\opP}(N) \otimes_{\bbk} U_{\opP}(B) \simto U_{\opP}(A)$.
\item
The two linear isomorphisms in the previous item 
are continuous in the following sense.
Let 
$\varphi:
 U_{\opP}(B) \otimes U_{\opP}(N) \simto 
 U_{\opP}(N) \otimes U_{\opP}(B)$
be the composition of the isomorphisms. 
Then that for every $m,n \in \bbZ$ there exist 
$k_+,k_- \in \bbN$ such that 
$\varphi \bigl(U_{\opP}(B)_m \otimes U_{\opP}(N)_n \bigr) 
 \subset \oplus_{k_- \le k \le k_+} 
 U_{\opP}(N)_{n-k} \otimes U_{\opP}(B)_{m+k}$.
We demand the same condition for the other composition 
$U_{\opP}(N) \otimes U_{\opP}(B) \simto 
 U_{\opP}(B) \otimes U_{\opP}(N)$.
\end{enumerate}
\end{dfn}

\begin{eg*}
Recall that enveloping algebras for $\opP=\opAsc,\opLie$
are nothing but the classical enveloping algebras
by Example \ref{eg:env-alg}.
Then one can immediately check that 
\begin{enumerate}
\item
The semi-infinite structure of an associative algebra 
given in \cite{A2} gives a semi-infinite structure in our sense
for $\opP=\opAsc$.
\item
The semi-infinite structure of a Lie algebra 
given in \cite{V1} gives a semi-infinite structure 
for $\opP=\opLie$.
\end{enumerate}
\end{eg*}

\subsection{Standard semi-injective resolution}

Let $A$ be an $\opP$-algebra in $\catE$ 
with a semi-infinite structure $B \inj A$. 
We will use the symbol $N:=\Coker(B \inj A)$.

\begin{dfn*}
For an $A$-module (or a left $U_{\opP}(A)$-module) 
$M$  in $\catE$, we set
\[
 \BD{\opP,\sinf+\bullet}(A,M) := 
 \Hom_{\cMod{\opP}{A}}(\BD{\opP}(A,B,A),M) \otimes_{U_{\opP}(A)}
 \BD{\opP}(A,N,A).
\]
The \emph{operadic semi-infinite homology} of $A$ with coefficient in $M$ 
is defined to be the homology of the complex 
$\BD{\opP,\sinf+\bullet}(A,M)$.
\end{dfn*}

By the argument for associative algebras 
in \cite[Proposition 2.6.3]{Se},
we have

\begin{thm}
$\BD{\sinf+\bullet}(A,M)$ is a semi-injective resolution of $M$,
In other words, it is a complex of $A$-modules satisfying
\begin{enumerate}
\item 
$K$-injective as a complex of $N$-modules, 
\item
$K$-projective relative to $N$.
\end{enumerate}
\end{thm}

By construction, setting $\opP=\opAsc$ or $\opP=\opLie$,
we recover the complexes appearing in the literature \cite{Fe,V1,A2,Se}.
In particular we have 

\begin{thm}
For the classical operad $\op=\opAsc$ or $\opP=\opLie$, 
the homology of the complex 
$\BD{\opP,\sinf+\bullet}(A,M)$ gives 
the semi-infinite cohomology of $A$ with coefficient in $M$
in the literature.
\end{thm}


\end{document}